\documentclass[10pt]{article}
\usepackage{amsmath}
\usepackage{amsfonts}
\usepackage{amssymb}
\usepackage{graphicx}
\usepackage[T1]{fontenc}
\usepackage{mathrsfs}
\usepackage{xcolor}
\usepackage{bbm}
\usepackage{verbatim}
\usepackage{dsfont}
\usepackage{mathrsfs}
\usepackage[body={15.5cm,21cm}, top=3cm]{geometry}
\usepackage{paralist}
\usepackage{indentfirst}
\allowdisplaybreaks[4]
\usepackage{hyperref}
\hypersetup{hypertex = true,
colorlinks=true,
linkcolor=blue,
anchorcolor=blue,
citecolor =blue }
\providecommand{\U}[1]{\protect \rule{.1in}{.1in}}
\numberwithin{equation}{section}
\newtheorem{Theorem}{Theorem}[section]

\newtheorem{lemma}[Theorem]{Lemma}

\newtheorem{remark}[Theorem]{Remark}

\newenvironment{proof}[1][Proof]{\noindent \textbf{#1.} }{\  \rule{0.5em}{0.5em}}
\allowdisplaybreaks[2]
\begin{document}
\title{Well-posedness for a class of mean field type FBSDEs and classical solutions of related master equations}
\author{Tianjiao Hua \thanks{School of Mathematical Sciences, Shanghai Jiao Tong University, China (htj960127@sjtu.edu.cn)}
\and
Peng Luo \thanks{School of Mathematical Sciences, Shanghai Jiao Tong University, China (peng.luo@sjtu.edu.cn). Research supported
by the National Natural Science Foundation of China (No. 12101400).}}
\maketitle

\begin{abstract}
In this paper, we study a class of mean field type FBSDEs. We propose a class of motonotinity conditions, under which we show the uniformly Lipschitz continuity of the decoupling field and obtain the existence and uniqueness of solution. We further provide a representation result for the solution and the decoupling field. Finally, we obtain the regularity of the decoupling field and establish global well-posedness of classical solutions to related master equations. 
\end{abstract}

\noindent\textbf{MSC-classification}: 60H30, 49N80.
\section{Introduction}

In the past few years, accompanying the research popularity of mean field games led by Huang et al. \cite{1183728987} and Lasry and Lions \cite{lasry2007mean}, mean field type FBSDEs gained a rapid development. Using stochastic maximum principle, mean field type FBSDE systems naturally arise when dealing with mean field game problems and mean field control problems (see \cite{carmona2013probabilistic,10.1214/14-AOP946}). In particular, for linear-quadratic (LQ for short) mean field game problems and mean field control problems where the mean field interaction is only through the expectations of the state and control, the Hamiltonian systems turn to be the following mean field type FBSDEs \eqref{FBSDE}, see Carmona and Delarue \cite{carmona2018probabilistic}, Acciaio et al. \cite{doi:10.1137/18M1196479}, Tian and Yu \cite{tian2023mean}, and so on. 

Motivated by this, in this paper, we consider the following mean field type FBSDEs:
\begin{equation}
    \left\{\begin{aligned}
X_t&=x+\int_0^t \left[b_{1}(s)X_{s}+b_{2}(s)Y_{s}+ b_{0}\left(s,\mathbb{E}[X_s],\mathbb{E}[Y_s]\right) \right] d s+\int_0^t \sigma\left(s, X_s, Y_s,\mathbb{E}[X_s],\mathbb{E}[Y_s]\right) d W_s, \\
Y_t&=h_{1}X_{T}+h_{2}\left(\mathbb{E}[X_T]\right)+\int_t^T\left[f_{1}(s)X_{s}+f_{2}(s)Y_{s}+f_{0}\left(s,\mathbb{E}[X_s],\mathbb{E}[Y_s]\right)\right]d s-\int_t^T Z_sd W_s.
\end{aligned}\right.\label{FBSDE}
\end{equation}
From both theoretical and practical perspectives, it is appealing to study the global solution to coupled FBSDEs. However, it is well known that classical coupled FBSDEs may be unsolvable on an arbitrarily large time interval only under the uniformly Lipschitz continuity of the coefficients and a counterexample was given in Antonelli \cite{antonelli1993backward}. Hence, some more assumptions need to be imposed to ensure global solvability. As for mean field FBSDEs, the additional mean field term brings new difficulty and solvability results are usually based on expansion and development of methods for FBSDEs. Carmona and Delarue \cite{carmona2013mean} obtained an existence result on arbitrarily large time interval for the first time by combining the result about FBSDEs in \cite{delarue2002existence} and Schauder's fixed point theorem applied to probability measures. Moreover, Bensoussan et al. \cite{bensoussan2015well} extended the continuation method originally introduced by Hu and Peng \cite{hu1995solution}, Yong \cite{yong1997finding}, and Peng and Wu \cite{peng1999fully} to study mean field FBSDEs by adjusting the corresponding   monotonicity conditions. Recently, adopting continuation method,  domination-monotonicity conditions are introduced to solve mean field type FBSDEs in \cite{tian2023mean}. 
 
 In this paper, we aim to establish the well-posedness of the solution to the mean field FBSDE \eqref{FBSDE} under some additional monotonicity conditions on the coefficients. Our approach is a development of the decoupling field method originally introduced by Ma et al. \cite{ma2015well} for classical coupled FBSDEs. It is worth emphasizing that getting a uniformly Lipschitz constant of the decoupling field is the key ingredient to extend the solution from small duration to arbitrarily large time interval. In the literature, the decoupling field related to mean field FBSDEs usually includes a probability measure as one of the state variables (see \cite{chassagneux2022probabilistic}). Thus one needs to establish the uniformly Lipschitz continuity on the Wassertein space, which is quite difficult in general. Relying on the structure of mean field type FBSDE \eqref{FBSDE}, where the mean field interaction is only through the expectations, we first establish some stability results. Based on these results, we are able to introduce the decoupling field by including the expectation of a probability measure as one of the state variables. Under some monotonicity conditions, we further establish the uniformly Lipschitz continuity of the decoupling field. Finally, we obtain the existence and uniqueness of global solution of mean field type FBSDE \eqref{FBSDE}. We emphasize that our approach does not require the non-degeneracy of the diffusion process and can solve some mean field type FBSDEs which can not be solved by the continuation method (see Subsection \ref{sec:com} for more discussions). In particular, we provide a new situation where the general assumption (H3) in \cite{chassagneux2022probabilistic} is ensured (see Remark \ref{rem:del}).

In addition to the well-posedness result, we obtain a representation result for the unique solution of mean field FBSDE \eqref{FBSDE} and the corresponding decoupling field. This further helps us to obtain the regularity of the decoupling field and establish global well-posedness of classical solutions to related master equations. When the data are sufficient smooth, the master equation usually admits a classical solution on a small time horizon, see, e.g. \cite{carmona2018probabilistic,bensoussan2019control,cardaliaguet2022splitting,gangbo2015existence}. However, the global well-posedness of classical solutions to the master equations is more challenging. In the literature, global well-posedness of the master equation is usually established under three types of monotonicity conditions: the well-known Lasry-Lions monotonicity condition, see e.g. \cite{cardaliaguet2019master,lasrymaster,carmona2018probabilistic,chassagneux2022probabilistic,wellposednesszhang,bayraktar2018analysis}; displacement monotonicity condition, see e.g. \cite{gangbo2022global,gangbo2022mean}; anti-monotonicity condition, see e.g. \cite{mou2022mean}. It is worth emphasizing that our monotonicity condition is quite distinct and there is no necessary implication from our monotonicity condition to  the above three kind of monotonicity conditions. We show that the decoupling field we define for mean field FBSDE \eqref{FBSDE} is indeed a globally classical solution of the corresponding master equation. The relationship of the decoupling field generated by mean field FBSDEs with the classical solution of the corresponding master equations was established in \cite{chassagneux2022probabilistic}. Compared with \cite{chassagneux2022probabilistic}, we establish the smoothness of the decoupling field with different techniques and we do not require the diffusion process to be bounded. More recently, Li et al. \cite{li2023linear} also investigate a similar type of master equations arising from mean field games of controls and obtained global well-posedness. The non-degeneracy of the common noise is critical in their work. By contrast, we consider the situation without common noise.

 The rest of the paper is organized as follows. In section 2, we introduce some notations. In section 3, we establish the well-posedness of mean field FBSDEs \eqref{FBSDE} and compare our results with some existing results. In section 4, we provide a representation result for mean field FBSDEs \eqref{FBSDE}, while obtain global well-posedness of classical results of related master equations in section 5.
\section{Notations}

 For a given $T>0$, let $(\Omega, \mathcal{F},(\mathcal{F}_{t})_{0\leq t \leq T}, \mathbb{P})$ be a completed filtered probability space with $(\mathcal{F}_{t})_{0\leq t \leq T}$ generated by a $d$-dimensional Brownian motion $\left(W_t\right)_{t \geqslant 0}$. Unless otherwise stated, all equalities and inequalities between random variables and processes will be understood in the $\mathbb{P}$-a.s. and $\mathbb{P} \otimes d t$-a.e. sense, respectively. $|\cdot|$ denotes the Euclidean norm and $x\cdot y$ denotes the inner product of $x$ and $y$. For $x, y \in \mathbb{R}^{n}, x \leq y$ is understood component-wisely, i.e., $x \leq y$ if and only if $x^{i} \leq y^{i}$ for all $i=1, \ldots, n$.
$C^k(\mathbb{R}^{n};\mathbb{R}^k)$ denotes the space of all $\mathbb{R}^k$-valued and continuous functions $f$ on $\mathbb{R}^{n}$ with continuous derivatives up to order $k$.
$C^0([0,T] \times \mathbb{R}^{n};\mathbb{R}^k)$ denotes the space of all $\mathbb{R}^k$-valued and continuous functions $f$ on $[0,T] \times \mathbb{R}^{n}$.
$C^{1,2}([0, T] \times \mathbb{R}^{n};\mathbb{R}^k)$ denotes the space of  all $\mathbb{R}^k$-valued and continuous functions $f$ on $[0, T] \times \mathbb{R}^{n}$ whose partial derivatives $\frac{\partial f}{\partial t}, \frac{\partial f}{\partial x_i}, \frac{\partial^2 f}{\partial x_i \partial x_j}, 1 \leq i, j \leq n$, exist and are continuous. For $p \ge 1$, we introduce some Banach spaces as follows:
\begin{itemize}
	\item $\mathbb{S}^{\infty}\left(\mathbb{R}^{n}\right)$ the set of $n$-dimensional continuous adapted processes $Y$ on $[0, T]$ such that
	\begin{equation*}
		\|Y\|_{\mathbb{S}^{\infty}\left(\mathbb{R}^{n}\right)}:=\left\|\sup _{0 \leq t \leq T}\left|Y_{t}\right|\right\|_{\infty}<\infty ;
	\end{equation*}
	\item $L^{p}\left(\mathcal{F}_{t} ; \mathbb{R}^{n}\right)$ the set of $n$-dimensional $\mathcal{F}_{t}$-measurable random variables $\xi$ such that
	\begin{equation*}
		E\left[ |\xi|^{p}\right]^{\frac{1}{p}}<\infty;
	\end{equation*}
	\item $L^{\infty}\left(\mathcal{F}_{t} ; \mathbb{R}^{n}\right)$ the set of $n$-dimensional $\mathcal{F}_{t}$-measurable random variables $\xi$ such that
	\begin{equation*}
		\|\xi\|_{\infty}<\infty;
	\end{equation*}
	\item  $\mathbb{S}^{p}\left(\mathbb{R}^{n}\right)$ the set of adapted and continuous processes $X$ valued in $\mathbb{R}^{n}$ such that
	\begin{equation*}
		\|X\|_{\mathbb{S}^{p}\left(\mathbb{R}^{n}\right)}^{p}:=E\left[\sup _{0 \leq t \leq T} \left|X_{t}\right|^{p}\right]<\infty;
	\end{equation*}
	\item $\mathbb{H}^{p}\left(\mathbb{R}^{n \times d}\right)$ the set of predictable processes $Z$ valued in $\mathbb{R}^{n \times d}$ such that
	\begin{equation*}
		\|Z\|_{\mathbb{H}^{p}\left(\mathbb{R}^{n \times d}\right)}^{p}:=E\left[\left(\int_{0}^{T}\left|Z_{u}\right|^{2} d u\right)^{p / 2}\right]<\infty.
	\end{equation*}
\end{itemize}

\indent In the sequel, we will use the notation $\mathcal{L}(\Theta)$ to denote the law of the random variable $\Theta$ and use the notation $\mathbb{E}[\Theta]$ to denote the expectation of the random vairable $\Theta$. Let $W_{2}$ denote 2-Wassertein's distance on $\mathcal{P}_{2}(\mathbb{R}^{n})$ defined by 
\begin{equation*}
    W_{2}(\mu_{1},\mu_{2}) \triangleq \inf\left\{  \left[    \int_{\mathbb{R}^{n}\times \mathbb{R}^{n}}|x-y|^{2}\pi(dx,dy)\right]^{\frac{1}{2}}\pi\in \mathcal{P}_{2}(\mathbb{R}^{n}\times \mathbb{R}^{n}) \text{ with marginals } \mu_{1} \text{ and }\mu_{2}    \right\}.
\end{equation*}
Throughout the paper, for any $x \in \mathbb{R}$ and any function $\phi(x)$, we will use the following convention
 \begin{equation*}
    \frac{\phi(x)-\phi(x)}{x-x}:=0.
\end{equation*}
\section{Solvability for a class of mean field type FBSDEs}
In this paper, we consider the following class of mean field type FBSDEs:
\begin{equation}
    \left\{\begin{aligned}
X_t&=x+\int_0^t \left[b_{1}(s)X_{s}+b_{2}(s)Y_{s}+ b_{0}\left(s,\mathbb{E}[X_s],\mathbb{E}[Y_s]\right) \right] d s+\int_0^t \sigma\left(s, X_s, Y_s,\mathbb{E}[X_s],\mathbb{E}[Y_s]\right) d W_s, \\
Y_t&=h_{1}X_{T}+h_{2}\left(\mathbb{E}[X_T]\right)+\int_t^T\left[f_{1}(s)X_{s}+f_{2}(s)Y_{s}+f_{0}\left(s,\mathbb{E}[X_s],\mathbb{E}[Y_s]\right)\right]d s-\int_t^T Z_sd W_s.
\end{aligned}\right.\label{mean field FBSDE}
\end{equation}
    This class of mean field type FBSDEs naturally appear in price impact problems (see \cite{carmona2018probabilistic}) and mean field linear quadratic optimal control problems (see \cite{doi:10.1137/18M1196479,tian2023mean}). In the sequel, we will study the existence and uniqueness theorem for \eqref{mean field FBSDE}. Let $K$ be a given positive constant, we introduce the following assumptions.
    \begin{itemize}
        \item[\textbf{(A1)}] (i) The mappings $b_{1}: [0, T] \rightarrow  \mathbb{R}$, $b_{2},f_{1}: [0, T] \rightarrow  \mathbb{R}^{n}$, $f_{2}: [0, T] \rightarrow  \mathbb{R}^{n\times n}$ are deterministic measurable functions and bounded by $K$ and $b_{0}(\omega,t,\bar{x},\bar{y}):\Omega\times[0,T]\times \mathbb{R} \times \mathbb{R}^{n}\rightarrow \mathbb{R}$ and $f_{0}(\omega,t,\bar{x},\bar{y}):\Omega\times[0,T]\times \mathbb{R} \times \mathbb{R}^{n}\rightarrow \mathbb{R}^{n}$ are measurable and uniformly Lipschitz continuous with respect to $\bar{x}$ and $\bar{y}$, i.e., 
         \begin{equation}
              \begin{aligned}
         \left|b_{0}(t,\bar{x}, \bar{y})-b_{0}\left(t,\bar{x}^{\prime},\bar{y}^{\prime}\right)\right| & \leq K\left[\left|\bar{x}-\bar{x}^{\prime}\right|+ \left|\bar{y}-\bar{y}^{\prime}\right|\right],\\
         \left|f_{0}(t,\bar{x}, \bar{y})-f_{0}\left(t,\bar{x}^{\prime},\bar{y}^{\prime}\right)\right| & \leq K\left[\left|\bar{x}-\bar{x}^{\prime}\right|+ \left|\bar{y}-\bar{y}^{\prime}\right|\right].
        \end{aligned}
            \end{equation}
            (ii) $\sigma: \Omega\times [0,T]\times \mathbb{R} \times \mathbb{R}^{n} \times \mathbb{R} \times \mathbb{R}^{n} \rightarrow \mathbb{R}^{n\times d}$ is measurable and uniformly Lipschitz continuous with respect to all the variables, i.e., 
            \begin{equation}
         \left|\sigma(t, x, y, \bar{x}, \bar{y})-\sigma\left(t, x^\prime,y^\prime,\bar{x}^{\prime},\bar{y}^{\prime}\right)\right|  \leq K\left[\left|x-x^{\prime}\right|+\left|y-y^{\prime}\right|+\left|\bar{x}-\bar{x}^{\prime}\right|+ \left|\bar{y}-\bar{y}^{\prime}\right|\right].
            \end{equation}
        (iii)$h_{1}$ is a n-dimensional constant vector  bounded by $K$ and $h_{2}:\Omega \times \mathbb{R}\rightarrow \mathbb{R}^{n}$ is measurable and uniform Lipschitz continuous, i.e. 
            \begin{equation*}
                \left| h_{2}(\bar{x})-h_{2}(\bar{x}^{\prime}) \right|\leqslant K\left( \left|\bar{x}-\bar{x}^{\prime} \right|\right).
            \end{equation*}
            (iiii)
            The following integrability condition holds:
    \begin{equation*}
       \mathbb{E}\left[\left( \int_{0}^{T}|b_{0}(t,0,0)|dt\right)^{2}+\left(\int_{0}^{T}|f_{0}(t,0,0)|dt\right)^{2}+\int_{0}^{T}|\sigma_{0}(t,0,0,0,0)|^2dt+|h_{2}(0)|^{2}\right]<\infty.
    \end{equation*}
         \end{itemize}
  In order to obtain the global solvability for mean field type FBSDE \eqref{mean field FBSDE}, we introduce the following notations. For ease of notations, for $y_{1},y_{2}\in \mathbb{R}^{n}$, we denote 
\begin{equation*}
    (y_{1}^{(1,i)},y_{2}^{(i+1,n)}):=(y_{1}^{1},y_{1}^{2},\cdots,y_{1}^{i},y_{2}^{i+1},\cdots,y_{2}^{n}).
\end{equation*}
For $\left(\bar{x}_{1}, \bar{y}_{1}\right),\left(\bar{x}_{2}, \bar{y}_{2}\right)\in\mathbb{R}\times\mathbb{R}^n$, let $\theta_1:=\left(\bar{x}_{1}, \bar{y}_{1}\right),\theta_2:=\left(\bar{x}_{2}, \bar{y}_{2}\right)$ and for $i,j=1,2,\cdots,n$, we denote
\begin{equation}
	\begin{aligned}
		h_{2}^{i}(\bar{x}_{1},\bar{x}_{2}) &\triangleq \frac{h^{i}_{2}(\bar{x}_{1})-h^{i}_{2}(\bar{x}_{2})}{\bar{x}_{1}-\bar{x}_{2}},\\
	b_{3}(t,\theta_{1},\theta_{2})&\triangleq \frac{b_{0}(t,\bar{x}_{1},\bar{y}_{1})-b_{0}(t,\bar{x}_{2},\bar{y}_{1})}{\bar{x}_{1}-\bar{x}_{2}},\\
		f_{3}^{i}(t,\theta_{1},\theta_{2}) &\triangleq \frac{f^{i}_{0}(t,\bar{x}_{1},\bar{y}_{1})-f^{i}(t,\bar{x}_{2},\bar{y}_{1})}{\bar{x}_{1}-\bar{x}_{2}},\\
	b_{4}^{j}(t,\theta_{1},\theta_{2}) &\triangleq \frac{b_{0}(t,\bar{x}_{2},\bar{y}_{2}^{(1,j-1)},\bar{y}_{1}^{(j,n)})-b_{0}(t,\bar{x}_{2},\bar{y}_{2}^{(1,j)},\bar{y}_{1}^{(j+1,n)})}{\bar{y}_{1}^{j}-\bar{y}_{2}^{j}},\\
	f_{4}^{ij}(t,\theta_{1},\theta_{2})&\triangleq \frac{f^{i}_{0}(t,\bar{x}_{2},\bar{y}_{2}^{(1,j-1)},\bar{y}_{1}^{(j,n)})-f^{i}_{0}(t,\bar{x}_{2},\bar{y}_{2}^{(1,j)},\bar{y}_{1}^{(j+1,n)})}{\bar{y}_{1}^{j}-\bar{y}_{2}^{j}},
		\end{aligned} \label{quadratic notation}
\end{equation}
and $b_{4}(t,\theta_{1},\theta_{2}) \triangleq (b_{4}^{1},b_{4}^{2},\cdots,b_{4}^{n})(t,\theta_{1},\theta_{2})$, $f_{4}^{i}(t,\theta_{1},\theta_{2}) \triangleq (f_{4}^{i1},f_{4}^{i2},\cdots,f_{4}^{in})(t,\theta_{1},\theta_{2})$. With these notations at hand, we make the following monotonicity assumptions.
\begin{itemize}
\item[\textbf{(A2)}]
    For $1\leq i,j\leq n$ satisfying $i\neq j$ and $t\in[0,T]$, one of the following holds:\\
    (i)  $f_{1}^{i}(t) \geq 0$, $h_{1}^{i} \geq 0$, $b_{2}^{i}(t)\leq 0$ and $f_{2}^{ij}(t)\geq 0$.\\
 (ii) $f_{1}^{i}(t) \leq 0$, $h_{1}^{i} \leq 0$, $b_{2}^{i}(t)\geq 0$ and $f_{2}^{ij}(t)\geq 0$.\\
 \item[\textbf{(A3)}] For $1\leq i,j \leq n $ satisfying $i\neq j$, $t\in[0,T]$ and any $\theta_{1}=(\bar{x}_{1},\bar{y}_{1}),\theta_{2}=(\bar{x}_{2},\bar{y}_{2}) \in \mathbb{R}\times \mathbb{R}^{n}$, one of the following holds:\\
    (i) $f_{1}^{i}(t)+\mathbb{E}[f_{3}^{i}(t,\theta_{1},\theta_{2})]\geq 0, h_{1}^{i}+\mathbb{E}[h_{2}^{i}(\bar{x}_{1},\bar{x}_{2})] \geq 0$, $b_{2}^{i}(t)+\mathbb{E}[b_{4}^{i}(t,\theta_{1},\theta_{2})]\leq 0 $ and  $f^{ij}_{2}(t)+\mathbb{E}[f^{ij}_{4}(t,\theta_{1},\theta_{2})]\geq 0$. \\
 (ii) $f_{1}^{i}(t)+\mathbb{E}[f_{3}^{i}((t,\theta_{1},\theta_{2})]\leq 0, h_{1}^{i}+\mathbb{E}[h_{2}^{i}(\bar{x}_{1},\bar{x}_{2})] \leq 0$ , $b_{2}^{i}(t)+\mathbb{E}[b_{4}^{i}(t,\theta_{1},\theta_{2})]\geq 0 $ and  $f^{ij}_{2}(t)+\mathbb{E}[f^{ij}_{4}(t,\theta_{1},\theta_{2})]\geq 0$.
\end{itemize}
\begin{remark}
It is obvious that $f_{1}^{i}+f_{3}^{i}\geq 0$ implies that $f_{1}^{i}+\mathbb{E}[f_{3}^{i}] \geq 0$. In particular they are equivalent when $f_0$ is deterministic. Similar arguments hold for other inequalities in assumption (A3).
\end{remark}

\indent  The main result of this section is on the existence and uniqeness of global solution of mean field type FBSDE \eqref{mean field FBSDE}. 
\begin{Theorem}
Under assumptions $(A1)-(A3)$, mean field type FBSDE \eqref{mean field FBSDE}  has a unique solution $(X,Y,Z) \in \mathbb{S}^{2}(\mathbb{R}) \times \mathbb{S}^{2}(\mathbb{R}^{n})\times \mathbb{H}^{2}(\mathbb{R}^{n\times d})$.\label{global existence}
\end{Theorem}
The key idea to get the global solution is to establish some crucial estimates and use a pasting technique. Following this idea, we first introduce the following mean field type FBSDE.
 \begin{equation}
   \left\{\begin{aligned}
X_s^{t,\eta}&=\eta+\int_t^s \left[b_{1}(r)X_r^{t,\eta}+b_{2}(r)Y_r^{t,\eta}+b_{0}\left(r,\mathbb{E}[X_r^{t,\eta}],\mathbb{E}[Y_r^{t,\eta}]\right)\right] d r\\
&\quad+\int_t^s \sigma\left(r, X_r^{t,\eta}, Y_r^{t,\eta},\mathbb{E}[X_r^{t,\eta}],\mathbb{E}[Y_r^{t,\eta}]\right) d W_r, \\
Y_s^{t,\eta}&=h_{1}X_T^{t,\eta}+h_{2}(\mathbb{E}[X_T^{t,\eta}])+\int_s^T\left[f_{1}(r)X_r^{t,\eta}+f_{2}(r)Y_r^{t,\eta}+f_{0}\left(r,\mathbb{E}[X_r^{t,\eta}],\mathbb{E}[Y_r^{t,\eta}]\right)\right]d r\\
&\quad -\int_s^T Z_r^{t,\eta} d W_r.
\end{aligned}\right. \label{mean field fbsde}
\end{equation}
We will make the following solvabilty assumption for \eqref{mean field fbsde}.\\ 
\textbf{Assumption (H).}
There exists $s\in[0,T)$ such that for any $t\in[s,T]$ and $\eta \in L^{2}(\Omega,\mathcal{F}_{t},\mathbb{P};\mathbb{R})$, FBSDE \eqref{mean field fbsde} has a unique solution $(X^{t,\eta},Y^{t,\eta},Z^{t,\eta}) \in \mathbb{S}^{2}(\mathbb{R}) \times \mathbb{S}^{2}(\mathbb{R}^{n})\times \mathbb{H}^{2}(\mathbb{R}^{n\times d})$.
\begin{remark}  
We remark that under assumption (A1), assumption (H) is satisfied with $s$ being sufficiently close to $T$ (see \cite[Theorem 4.24]{carmona2018probabilistic}).
\end{remark}

\indent  For the subsequent analysis, under assumption (H), we introduce, for any $x \in \mathbb{R}$, the following FBSDE (associated to the system \eqref{mean field fbsde}),
\begin{equation}
    \left\{\begin{aligned}
X_s^{t,x,\eta}&=x+\int_t^s \left[b_{1}(r)X_r^{t,x,\eta}+b_{2}(r)Y_r^{t,x,\eta}+b_{0}\left(r,\mathbb{E}[X_r^{t,\eta}],\mathbb{E}[Y_r^{t,\eta}]\right)\right] d r\\
&\quad+\int_t^s \sigma\left(r, X_r^{t,x,\eta}, Y_r^{t,x,\eta},\mathbb{E}[X_r^{t,\eta}],\mathbb{E}[Y_r^{t,\eta}]\right)d W_r, \\
 Y_s^{t,x,\eta}&=h_{1}X_T^{t,x,\eta}+h_{2}\left(\mathbb{E}[X_T^{t,\eta}]\right)+\int_s^T\left[f_{1}(r)X_r^{t,x,\eta}+f_{2}(r)Y_r^{t,x,\eta}+f_{0}\left(r,\mathbb{E}[X_r^{t,\eta}],\mathbb{E}[Y_r^{t,\eta}]\right)\right]
 d r\\
 &\quad -\int_s^T Z_r^{t,x,\eta} d W_r,
\end{aligned}\right. \label{fbsde}
 \end{equation}
where $X^{t,\eta}, Y^{t,\eta}$ are the first two components of the unique solution of \eqref{mean field fbsde}.\\
 We will first introduce the following two stability results, i.e, Lemma \ref{lemma} and Lemma \ref{small bound}. 
\begin{lemma}
     Suppose assumptions $(H)$ and $(A1)-(A2)$ hold, for any $t\in[s,T]$ and $\eta,\bar{\eta}\in L^{2}(\Omega,\mathcal{F}_{t},\mathbb{P};\mathbb{R})$, let $(X^{t,x,\eta},Y^{t,x,\eta},Z^{t,x,\eta})$ (resp. $(X^{t,x,\bar{\eta}},Y^{t,x,\bar{\eta}},Z^{t,x,\bar{\eta}})$) be the unique solution of FBSDE \eqref{fbsde} associated with $\eta$ (resp. $\bar{\eta}$), then it holds that
\begin{equation}
\begin{aligned}
&\mathbb{E}\left[\sup _{t \leq r \leq T}\left|X^{t,x,\eta}_r-X^{t,x,\bar{\eta}}_r\right|^2|\mathcal{F}_t\right]+\mathbb{E}\left[\sup _{t \leq r \leq T}\left|Y^{t,x,\eta}_r-Y^{t,x,\bar{\eta}}_r\right|^2|\mathcal{F}_t\right]+\mathbb{E} \left[\int_t^T\left|Z^{t,x,\eta}_r-Z^{t,x,\bar{\eta}}_r\right|^2 d r |\mathcal{F}_t\right]\\
&\quad \leq \Gamma \left(|\mathbb{E}[X^{t,\eta}_{T}]-\mathbb{E}[X^{t,\bar{\eta}}_{T}]|^{2}+\int_t^T |\mathbb{E}[X^{t,\eta}_{r}]-\mathbb{E}[X^{t,\bar{\eta}}_{r}]|^{2}+|\mathbb{E}[Y^{t,\eta}_{r}]-\mathbb{E}[Y^{t,\bar{\eta}}_{r}]|^{2} d r\right),
\end{aligned}\label{small stability}
\end{equation} 
where $ \Gamma $ depending on $n,K,T$ and is independent of $t$.\label{lemma}
\end{lemma}
\begin{proof}
Under assumptions (H) and (A1)-(A2), for any $t\in[s,T]$ and $\eta \in L^{2}(\Omega,\mathcal{F}_{t},\mathbb{P};\mathbb{R})$, the solvability of FBSDE \eqref{fbsde} on $[t,T]$ follows as a straightforward consequence of our previous work \cite{hua2022unified}. Moreover, for any $x_1,x_2\in\mathbb{R},t\in[s,T]$, it holds that 
\begin{equation*}
    |Y^{t,x_1,\eta}_t-Y^{t,x_2,\eta}_t|\leq \bar{C}|x_1-x_2|,
\end{equation*}
where $\bar{C}$ only depends on $n,K,T$ and is independent of $t$. Thus, the uniform stability result \eqref{small stability} follows directly from a combination of arguments from \cite[Theorem 1.3]{delarue2002existence} and \cite[Lemma 2.4]{carmona2013mean}.
\end{proof} 
\begin{lemma}
Under assumptions $(H),\ (A1)$ and $(A3)$, for any $t\in[s,T]$ and $\eta, \bar{\eta}\in L^{2}(\Omega,\mathcal{F}_{t},\mathbb{P};\mathbb{R})$, let $(X^{t,\eta},Y^{t,\eta},Z^{t,\eta})$ (resp. $(X^{t,\bar{\eta}},Y^{t,\bar{\eta}},Z^{t,\bar{\eta}})$) be the unique solution of FBSDE \eqref{mean field fbsde} associated with $\eta$ (resp. $\bar{\eta}$), we have the following estimate:
\begin{equation}
\sup _{t\leq r \leq T}\left|\mathbb{E}[X_r^{t,\eta}]-\mathbb{E}[X_{r}^{t,\bar{\eta}}]\right|^2+ \sup _{t\leq r \leq T}\left|\mathbb{E}[Y_r^{t,\eta}]-\mathbb{E}[Y_{r}^{t,\bar{\eta}}]\right|^2\leq C|\mathbb{E}[\eta]-\mathbb{E}[\bar{\eta}]|^{2},\label{expectation estimate}
\end{equation}
where $C$ only depends on $n,K,T$, and is independent of $t$.\label{small bound}
\end{lemma}
\begin{proof}
    Let $M=2nK(T+1)e^{(2n+2)KT}$ and $\delta=\frac{1}{2M(3M^2+4)}$. We will only give the proof under assumptions (H), (A1) and (A3)(i), since the proof is similar under assumptions (H), (A1) and (A3)(ii). The proof will be divided into several steps.

    \textbf{Step 1: }
    For any $t\in[s,T]$ and $\eta\in L^{2}(\Omega,\mathcal{F}_{t},\mathbb{P};\mathbb{R})$, taking expectation in \eqref{mean field fbsde} implies that 
    \begin{equation*}
        \left\{
        \begin{aligned}
        \mathbb{E}[X^{t,\eta}_{r}] &= \mathbb{E}[\eta]+ \int_{t}^{r}[b_{1}(u)\mathbb{E}[X^{t,\eta}_{u}]+b_{2}(u)\mathbb{E}[Y^{t,\eta}_{u}]+b_{0}(u,\mathbb{E}[X^{t,\eta}_{u}],\mathbb{E}[Y^{t,\eta}_{u}])]du,\\
        \mathbb{E}[Y^{t,\eta}_{r}] &= h_{1}\mathbb{E}[X^{t,\eta}_T]+h_{2}(\mathbb{E}[X^{t,\eta}_T])+\int_r^T\left[f_{1}(u)\mathbb{E}[X^{t,\eta}_u]+f_{2}(u)\mathbb{E}[Y^{t,\eta}_u]+f_{0}\left(u,\mathbb{E}[X^{t,\eta}_u],\mathbb{E}[Y^{t,\eta}_u]\right)\right]d u.
        \end{aligned}
        \right.
    \end{equation*}
    Therefore, we have 
    \begin{equation*}
        \left\{
        \begin{aligned}
        \mathbb{E}[X^{t,\eta}_{r}]-\mathbb{E}[X^{t,\bar{\eta}}_{r}] &= \mathbb{E}[\eta]-\mathbb{E}[\bar{\eta}]+ \int_{t}^{r}[b_{1}(u)\left(\mathbb{E}[X^{t,\eta}_{u}]-\mathbb{E}[X^{t,\bar{\eta}}_{u}]\right)+b_{2}(u)\left(\mathbb{E}[Y^{t,\eta}_{u}]-\mathbb{E}[Y^{t,\bar{\eta}}_{u}]\right)]du\\
        &\qquad\qquad +\int_{t}^{r}[b_{0}(u,\mathbb{E}[X^{t,\eta}_{u}],\mathbb{E}[Y^{t,\eta}_{u}])-b_{0}(u,\mathbb{E}[X^{t,\bar{\eta}}_{u}],\mathbb{E}[Y^{t,\bar{\eta}}_{u}])]du,\\
        \mathbb{E}[Y^{t,\eta}_{r}]-\mathbb{E}[Y^{t,\bar{\eta}}_{r}] &= h_{1}\left(\mathbb{E}[X^{t,\eta}_T]-\mathbb{E}[X^{t,\bar{\eta}}_T]\right)+h_{2}(\mathbb{E}[X^{t,\eta}_T])-h_{2}(\mathbb{E}[X^{t,\bar{\eta}}_T])\\
        &\qquad\qquad+\int_r^T\left[f_{1}(u)\left(\mathbb{E}[X^{t,\eta}_u]-\mathbb{E}[X^{t,\bar{\eta}}_u]\right)+f_{2}(u)\left(\mathbb{E}[Y^{t,\eta}_u]-\mathbb{E}[Y^{t,\bar{\eta}}_u]\right)\right]du\\
        &\qquad\qquad+\int_r^T\left[f_{0}\left(u,\mathbb{E}[X^{t,\eta}_u],\mathbb{E}[Y^{t,\eta}_u]\right)-f_{0}\left(u,\mathbb{E}[X^{t,\bar{\eta}}_u],\mathbb{E}[Y^{t,\bar{\eta}}_u]\right)\right]du.
        \end{aligned}
        \right.
    \end{equation*}
    Thus, it holds
    \begin{align*}
        &|\mathbb{E}[X^{t,\eta}_{r}]-\mathbb{E}[X^{t,\bar{\eta}}_{r}]|^2 \\
        &\leq |\mathbb{E}[\eta]-\mathbb{E}[\bar{\eta}]|^2+ 2\int_{t}^{r}|b_{1}(u)|\left|\mathbb{E}[X^{t,\eta}_{u}]-\mathbb{E}[X^{t,\bar{\eta}}_{u}]\right|^2du\\
        &\qquad\qquad+ 2\int_{t}^{r}|b_{2}(u)|\left|\mathbb{E}[X^{t,\eta}_{u}]-\mathbb{E}[X^{t,\bar{\eta}}_{u}]\right|\left|\mathbb{E}[Y^{t,\eta}_{u}]-\mathbb{E}[Y^{t,\bar{\eta}}_{u}]\right|du\\
        &\qquad\qquad +2\int_{t}^{r}\left|\mathbb{E}[X^{t,\eta}_{u}]-\mathbb{E}[X^{t,\bar{\eta}}_{u}]\right|\left|b_{0}(u,\mathbb{E}[X^{t,\eta}_{u}],\mathbb{E}[Y^{t,\eta}_{u}])-b_{0}(u,\mathbb{E}[X^{t,\bar{\eta}}_{u}],\mathbb{E}[Y^{t,\bar{\eta}}_{u}])\right|du\\
        &\leq |\mathbb{E}[\eta]-\mathbb{E}[\bar{\eta}]|^2+ 2K\int_{t}^{r}\left|\mathbb{E}[X^{t,\eta}_{u}]-\mathbb{E}[X^{t,\bar{\eta}}_{u}]\right|^2du\\
        &\qquad\qquad+ K\int_{t}^{r}\left(\left|\mathbb{E}[X^{t,\eta}_{u}]-\mathbb{E}[X^{t,\bar{\eta}}_{u}]\right|^2+\left|\mathbb{E}[Y^{t,\eta}_{u}]-\mathbb{E}[Y^{t,\bar{\eta}}_{u}]\right|^2\right)du\\
        &\qquad\qquad +2K\int_{t}^{r}\left|\mathbb{E}[X^{t,\eta}_{u}]-\mathbb{E}[X^{t,\bar{\eta}}_{u}]\right|^2du+K\int_{t}^{r}\left(\left|\mathbb{E}[X^{t,\eta}_{u}]-\mathbb{E}[X^{t,\bar{\eta}}_{u}]\right|^2+\left|\mathbb{E}[Y^{t,\eta}_{u}]-\mathbb{E}[Y^{t,\bar{\eta}}_{u}]\right|^2\right)du\\
        &\leq |\mathbb{E}[\eta]-\mathbb{E}[\bar{\eta}]|^2+ 6K\int_{t}^{r}\left|\mathbb{E}[X^{t,\eta}_{u}]-\mathbb{E}[X^{t,\bar{\eta}}_{u}]\right|^2du+2K\int_{t}^{r}\left|\mathbb{E}[Y^{t,\eta}_{u}]-\mathbb{E}[Y^{t,\bar{\eta}}_{u}]\right|^2du\\
        &\leq |\mathbb{E}[\eta]-\mathbb{E}[\bar{\eta}]|^2+ 3M\int_{t}^{r}\left|\mathbb{E}[X^{t,\eta}_{u}]-\mathbb{E}[X^{t,\bar{\eta}}_{u}]\right|^2du+M\int_{t}^{r}\left|\mathbb{E}[Y^{t,\eta}_{u}]-\mathbb{E}[Y^{t,\bar{\eta}}_{u}]\right|^2du.
    \end{align*}
    Hence, we get 
    \begin{equation}\label{eq:estimateX}
    \begin{aligned}
        &\sup_{t\leq r\leq T}|\mathbb{E}[X^{t,\eta}_{r}]-\mathbb{E}[X^{t,\bar{\eta}}_{r}]|^2 \\
        &\leq |\mathbb{E}[\eta]-\mathbb{E}[\bar{\eta}]|^2+ 3M(T-t)\sup_{t\leq r\leq T}\left|\mathbb{E}[X^{t,\eta}_{r}]-\mathbb{E}[X^{t,\bar{\eta}}_{r}]\right|^2+M(T-t)\sup_{t\leq r\leq T}\left|\mathbb{E}[Y^{t,\eta}_{r}]-\mathbb{E}[Y^{t,\bar{\eta}}_{r}]\right|^2.
    \end{aligned}
    \end{equation}
    On the other hand, it holds that
    \begin{align*}
        &\left|\mathbb{E}[Y^{t,\eta}_{r}]-\mathbb{E}[Y^{t,\bar{\eta}}_{r}]\right|^2 \\
        &\leq  2|h_{1}|^2\left|\mathbb{E}[X^{t,\eta}_T]-\mathbb{E}[X^{t,\bar{\eta}}_T]\right|^2+2\left|h_{2}(\mathbb{E}[X^{t,\eta}_T])-h_{2}(\mathbb{E}[X^{t,\bar{\eta}}_T])\right|^2\\
        &\qquad\qquad+2\int_r^T|f_{1}(u)|\left|\mathbb{E}[X^{t,\eta}_u]-\mathbb{E}[X^{t,\bar{\eta}}_u]\right|\left|\mathbb{E}[Y^{t,\eta}_u]-\mathbb{E}[Y^{t,\bar{\eta}}_u]\right|du\\
        &\qquad\qquad+2\int_r^T|f_{2}(u)|\left|\mathbb{E}[Y^{t,\eta}_u]-\mathbb{E}[Y^{t,\bar{\eta}}_u]\right|^2du\\
        &\qquad\qquad+2\int_r^T\left|\mathbb{E}[Y^{t,\eta}_u]-\mathbb{E}[Y^{t,\bar{\eta}}_u]\right|\left|f_{0}\left(u,\mathbb{E}[X^{t,\eta}_u],\mathbb{E}[Y^{t,\eta}_u]\right)-f_{0}\left(u,\mathbb{E}[X^{t,\bar{\eta}}_u],\mathbb{E}[Y^{t,\bar{\eta}}_u]\right)\right|du\\
        &\leq  4K^2\left|\mathbb{E}[X^{t,\eta}_T]-\mathbb{E}[X^{t,\bar{\eta}}_T]\right|^2+2K\int_r^T\left|\mathbb{E}[Y^{t,\eta}_u]-\mathbb{E}[Y^{t,\bar{\eta}}_u]\right|^2du\\
        &\qquad\qquad+K\int_r^T\left(\left|\mathbb{E}[X^{t,\eta}_u]-\mathbb{E}[X^{t,\bar{\eta}}_u]\right|^2+\left|\mathbb{E}[Y^{t,\eta}_u]-\mathbb{E}[Y^{t,\bar{\eta}}_u]\right|^2\right)du\\
        &\qquad\qquad+2K\int_r^T\left|\mathbb{E}[Y^{t,\eta}_u]-\mathbb{E}[Y^{t,\bar{\eta}}_u]\right|^2du+K\int_r^T\left(\left|\mathbb{E}[X^{t,\eta}_u]-\mathbb{E}[X^{t,\bar{\eta}}_u]\right|^2+\left|\mathbb{E}[Y^{t,\eta}_u]-\mathbb{E}[Y^{t,\bar{\eta}}_u]\right|^2\right)du\\
        &\leq 4K^2\left|\mathbb{E}[X^{t,\eta}_T]-\mathbb{E}[X^{t,\bar{\eta}}_T]\right|^2+6K\int_r^T\left|\mathbb{E}[Y^{t,\eta}_u]-\mathbb{E}[Y^{t,\bar{\eta}}_u]\right|^2du+2K\int_r^T\left|\mathbb{E}[X^{t,\eta}_u]-\mathbb{E}[X^{t,\bar{\eta}}_u]\right|^2du\\
        &\leq M^2\left|\mathbb{E}[X^{t,\eta}_T]-\mathbb{E}[X^{t,\bar{\eta}}_T]\right|^2+3M\int_r^T\left|\mathbb{E}[Y^{t,\eta}_u]-\mathbb{E}[Y^{t,\bar{\eta}}_u]\right|^2du+M\int_r^T\left|\mathbb{E}[X^{t,\eta}_u]-\mathbb{E}[X^{t,\bar{\eta}}_u]\right|^2du.
    \end{align*}
    Thus, we have
    \begin{equation}\label{eq:estimateY}
        \begin{aligned}
            &\sup_{t\leq r\leq T}\left|\mathbb{E}[Y^{t,\eta}_{r}]-\mathbb{E}[Y^{t,\bar{\eta}}_{r}]\right|^2 \\
            &\leq M^2\left|\mathbb{E}[X^{t,\eta}_T]-\mathbb{E}[X^{t,\bar{\eta}}_T]\right|^2+3M(T-t)\sup_{t\leq r\leq T}\left|\mathbb{E}[Y^{t,\eta}_r]-\mathbb{E}[Y^{t,\bar{\eta}}_r]\right|^2+M(T-t)\sup_{t\leq r\leq T}\left|\mathbb{E}[X^{t,\eta}_r]-\mathbb{E}[X^{t,\bar{\eta}}_r]\right|^2.
        \end{aligned}
    \end{equation}
    Combining \eqref{eq:estimateX} and \eqref{eq:estimateY} implies
    \begin{equation}\label{eq:estimateXY}
    \begin{aligned}
        &\sup_{t\leq r\leq T}\left|\mathbb{E}[X^{t,\eta}_r]-\mathbb{E}[X^{t,\bar{\eta}}_r]\right|^2+\sup_{t\leq r\leq T}\left|\mathbb{E}[Y^{t,\eta}_{r}]-\mathbb{E}[Y^{t,\bar{\eta}}_{r}]\right|^2 \\
            &\leq |\mathbb{E}[\eta]-\mathbb{E}[\bar{\eta}]|^2+M^2\left|\mathbb{E}[X^{t,\eta}_T]-\mathbb{E}[X^{t,\bar{\eta}}_T]\right|^2\\
            &\qquad\qquad+4M(T-t)\sup_{t\leq r\leq T}\left|\mathbb{E}[Y^{t,\eta}_r]-\mathbb{E}[Y^{t,\bar{\eta}}_r]\right|^2+4M(T-t)\sup_{t\leq r\leq T}\left|\mathbb{E}[X^{t,\eta}_r]-\mathbb{E}[X^{t,\bar{\eta}}_r]\right|^2\\
            &\leq \left(M^2+1\right)|\mathbb{E}[\eta]-\mathbb{E}[\bar{\eta}]|^2+(M^2+4)M(T-t)\sup_{t\leq r\leq T}\left|\mathbb{E}[Y^{t,\eta}_r]-\mathbb{E}[Y^{t,\bar{\eta}}_r]\right|^2\\
            &\qquad\qquad+(3M^2+4)M(T-t)\sup_{t\leq r\leq T}\left|\mathbb{E}[X^{t,\eta}_r]-\mathbb{E}[X^{t,\bar{\eta}}_r]\right|^2.
    \end{aligned}
    \end{equation}
    If $s\geq T-\delta$, it follows immediately from \eqref{eq:estimateXY} that for any $t\in[s,T]$,
    \begin{equation}\label{eq:smallLip}
        \sup_{t\leq r\leq T}\left|\mathbb{E}[X^{t,\eta}_r]-\mathbb{E}[X^{t,\bar{\eta}}_r]\right|^2+\sup_{t\leq r\leq T}\left|\mathbb{E}[Y^{t,\eta}_{r}]-\mathbb{E}[Y^{t,\bar{\eta}}_{r}]\right|^2 \leq 2\left(M^2+1\right)|\mathbb{E}[\eta]-\mathbb{E}[\bar{\eta}]|^2,
    \end{equation}
    which completes the proof.

    \textbf{Step 2:} If $s< T-\delta$, for any $t\in[s,T-\delta)$, it follows again from \eqref{eq:estimateXY} that for any $u\in[T-\delta,T]$,
    \begin{equation}\label{eq:estimateDelta}
        \sup_{u\leq r\leq T}\left|\mathbb{E}[X^{t,\eta}_r]-\mathbb{E}[X^{t,\bar{\eta}}_r]\right|^2+\sup_{u\leq r\leq T}\left|\mathbb{E}[Y^{t,\eta}_{r}]-\mathbb{E}[Y^{t,\bar{\eta}}_{r}]\right|^2 \leq 2\left(M^2+1\right)|\mathbb{E}[X^{t,\eta}_u]-\mathbb{E}[X^{t,\bar{\eta}}_u]|^2.
    \end{equation}
    Now we assume $\mathbb{E}[X_{T-\delta}^{t,\eta}]\neq \mathbb{E}[X_{T-\delta}^{t,\bar{\eta}}]$  and denote for $u\in[T-\delta,T]$,
    \begin{equation*}
     \nabla X_u=\frac{X_{u}^{t,\eta}-X_{u}^{t,\bar{\eta}}}{\mathbb{E}[X_{T-\delta}^{t,\eta}]-\mathbb{E}[X_{T-\delta}^{t,\bar{\eta}}]}, \quad   \nabla Y_u=\frac{Y_{u}^{t,\eta}-Y_{u}^{t,\bar{\eta}}}{\mathbb{E}[X_{T-\delta}^{t,\eta}]-\mathbb{E}[X_{T-\delta}^{t,\bar{\eta}}]}, \quad   \nabla Z_u=\frac{Z_{u}^{t,\eta}-Z_{u}^{t,\bar{\eta}}}{\mathbb{E}[X_{T-\delta}^{t,\eta}]-\mathbb{E}[X_{T-\delta}^{t,\bar{\eta}}]},
     \end{equation*}
     and 
     \begin{equation*}
       \nabla \sigma_{u} =\frac{\sigma\left(u, X_u^{t,\eta}, Y_u^{t,\bar{\eta}},\mathbb{E}[X_u^{t,\eta}],\mathbb{E}[Y_u^{t,\eta}]\right)-\sigma\left(u, X_u^{t,\eta}, Y_u^{t,\bar{\eta}},\mathbb{E}[X_u^{t,\bar{\eta}}],\mathbb{E}[Y_u^{t,\bar{\eta}}]\right)}{\mathbb{E}[X_{T-\delta}^{t,\eta}]-\mathbb{E}[X_{T-\delta}^{t,\bar{\eta}}]}. 
     \end{equation*}
    One can easily check that  $(\nabla X_{u},\nabla Y_{u}, \nabla Z_{u})_{T-\delta \leq u \leq T}$ satisfies the following variational FBSDE:
    \begin{equation}
        \left\{ \begin{array}{l}
            \nabla X_{u}=\frac{X_{T-\delta}^{t,\eta}-X_{T-\delta}^{t,\bar{\eta}}}{\mathbb{E}[X_{T-\delta}^{t,\eta}]-\mathbb{E}[X_{T-\delta}^{t,\bar{\eta}}]}+ \int_{T-\delta}^{u}\left(b_1\nabla X_r+b_2  \nabla Y_r+b_{3} \mathbb{E}[\nabla X_{r}]+b_4\mathbb{E}[\nabla Y_r]\right)dr+\int_{T-\delta}^{u}\nabla\sigma_{r}dW_{r},\\
             \nabla Y_{u}^{i} = h_{1}^{i}\nabla X_{T}+h_{2}^{i}\mathbb{E}[\nabla X_{T}] + \int_{u}^{T} \left(f_{1}^{i}\nabla X_{r}+f_{2}^{i}\nabla Y_{r}+f_3^{i}\mathbb{E}[\nabla X_{r}]+f_{4}^{i}\mathbb{E}[\nabla Y_{r}]\right)dr\\
             \qquad\qquad\qquad\qquad-\int_{u}^{T}\nabla Z_{r}dW_{r},\qquad\qquad\qquad i=1,2,\cdots,n,
        \end{array} \right. \label{multidimension variation FBSDE}
    \end{equation}
    where 
    \begin{align*}
        &h_{2}^{i} \triangleq h_{2}^{i}(\mathbb{E}[X_{T}^{t,\eta}],\mathbb{E}[X_{T}^{t,\bar{\eta}}]),~~b_{3}(r) \triangleq b_{3}(r,(\mathbb{E}[X_{r}^{t,\eta}],\mathbb{E}[Y_{r}^{t,\eta}]),(\mathbb{E}[X_{r}^{t,\bar{\eta}}],\mathbb{E}[Y_{r}^{t,\bar{\eta}}])),\\
        &b_{4}(r) \triangleq b_{4}(r,(\mathbb{E}[X_{r}^{t,\eta}],\mathbb{E}[Y_{r}^{t,\eta}]),(\mathbb{E}[X_{r}^{t,\bar{\eta}}],\mathbb{E}[Y_{r}^{t,\bar{\eta}}])),~~f_{3}(r) \triangleq f_{3}(r,(\mathbb{E}[X_{r}^{t,\eta}],\mathbb{E}[Y_{r}^{t,\eta}]),(\mathbb{E}[X_{r}^{t,\bar{\eta}}],\mathbb{E}[Y_{r}^{t,\bar{\eta}}])),\\
        &f_{4}(r) \triangleq f_{4}(r,(\mathbb{E}[X_{r}^{t,\eta}],\mathbb{E}[Y_{r}^{t,\eta}]),(\mathbb{E}[X_{r}^{t,\bar{\eta}}],\mathbb{E}[Y_{r}^{t,\bar{\eta}}])).
    \end{align*}
    Taking expectation in \eqref{multidimension variation FBSDE} yields
    \begin{equation}
        \left\{ \begin{array}{l}
           \mathbb{E}[\nabla X_{u}] =1+\int_{T-\delta}^{u} \left((b_1 +\mathbb{E}[b_{3}])\mathbb{E}[\nabla X_{r}]+(b_2+\mathbb{E}[b_{4}]) \mathbb{E}[\nabla Y_{r}]\right)dr, \\
             \mathbb{E}[\nabla Y_{u}^{i}] = (h_{1}^{i}+\mathbb{E}[h_{2}^{i}])\mathbb{E}[\nabla X_{T}]+\int_{u}^{T}\left(  (f_{1}^{i}+\mathbb{E}[f_{3}^{i}])\mathbb{E}[\nabla X_{r}]+(f_{2}^{i}+\mathbb{E}[f_{4}^{i}])\mathbb{E}[\nabla Y_{r}]\right)dr,~i=1,2\cdots,n.
        \end{array}\right.\label{multidimension variation}
    \end{equation}
    We will now show that $\mathbb{E}[\nabla X_u]$ remains positive on the time interval $[T-\delta, T]$. To this end, let $\tau \in[T-\delta, T]$ be a stopping time such that $\mathbb{E}[\nabla X_u]$ is positive on $[T-\delta, \tau]$.
    Denote
    \begin{equation*}
    \tilde{Y}_u^{i}=\mathbb{E}[\nabla Y_u^{i}]\left[\mathbb{E}[\nabla X_u]\right]^{-1}, i=1,2,\cdots,n, \quad u \in[T-\delta, \tau].
    \end{equation*}
    Noting that on $[T-\delta,\tau]$,
    \begin{equation}
    \begin{aligned}
      d(\mathbb{E}[\nabla X_{r}])^{-1} &= -(\mathbb{E}[\nabla X_{r}])^{-2}\left[(b_1 +\mathbb{E}[b_{3}])\mathbb{E}[\nabla X_{r}]+(b_2+\mathbb{E}[b_{4}])\mathbb{E}[\nabla Y_{r}]\right]dr\\
      & = -(\mathbb{E}[\nabla X_{r}])^{-1}\left[b_{1}+\mathbb{E}[b_{3}]+(b_{2}+ \mathbb{E}[ b_{4}])\tilde{Y}_{r}\right]dr,
    \end{aligned}
      \end{equation}
    one can get
    \begin{equation}
        \mathbb{E}[\nabla X_{u}]^{-1}  = \exp\left(\int_{T-\delta}^{u}(-b_{1}-\mathbb{E}[b_{3}]-(b_{2}+\mathbb{E}[b_{4}]) \tilde{Y}_{r})dr\right),~~~~u\in[T-\delta,\tau].
        \end{equation}
     According to \eqref{eq:estimateDelta}, $\tilde{Y}_{s}$ is bounded on $[T-\delta,T]$ and the Lipschitz continuity of $b$ impies that $b_1,b_2,b_3,b_4$ are uniformly bounded. Therefore, $\mathbb{E}[\nabla X^{-1}]$ is bounded on $\left[T-\delta, \tau\right]$ by a positive constant that does not depend on $\tau$. This implies that $\nabla X $ can never reach $0$, i.e., we can choose $\tau=T$. Thefore, it holds on $[T-\delta,T]$ that for $i=1,2,\cdots,n$,
    \begin{equation*}
    \begin{aligned}
    & d \mathbb{E}\left[\nabla Y_r^{i}\right] \mathbb{E}\left[\nabla X_r\right]^{-1} \\
    & =\left(\left(-f_1^{i}-\mathbb{E}[f_3^{i}]\right)-\left(f_2^{i}+\mathbb{E}[f_4^{i}]\right)\tilde{Y}_r-\tilde{Y}_r^{i}\left(\left(b_1+\mathbb{E}[b_3]\right)+\left(b_2+\mathbb{E}[b_4]\right)\tilde{Y}_r\right)\right) d r \\
    & =-\left[\left(f_1^{i}+\mathbb{E}[f_3^{i}]\right)+\left(b_1+\mathbb{E}[b_3]\right)\tilde{Y}_{s}^{i}+\left(f_2^{i}+\mathbb{E}[f_4^{i}]\right) \tilde{Y}_{r}+\left(b_2+\mathbb{E}[b_4]\right) \tilde{Y}_r\tilde{Y}^{i}_{r}\right] d r.
    \end{aligned}
    \end{equation*}
    For each $i=1,2,\cdots,n$ and $r\in[T-\delta,T],~y\in\mathbb{R}^n$, we introduce the following two functions:
    \begin{equation*}
        \begin{aligned}
         H^{i}(r,y) &= \left(f_1^{i}+\mathbb{E}[f_3^{i}]\right)+\left(b_1+\mathbb{E}[b_3]\right)y^{i}+\left(f_2^{i}+\mathbb{E}[f_4^{i}]\right)y+\left(b_2+\mathbb{E}[b_4]\right) \tilde{Y}_ry^{i}\\
         &=\left(f_1^{i}+\mathbb{E}[f_3^{i}]\right)+\left(b_1+\mathbb{E}[b_3]\right)y^{i}+\sum_{j=1}^{n}(f_{2}^{ij}+\mathbb{E}[f_{4}^{ij}])y^{j}+\sum_{j=1}^{n}(b_{2}^{j}+\mathbb{E}[b_{4}^{j}])\tilde{Y}^{j}_ry^i,\\
         \underline{H}^{i}(r,y)&=\left(b_1+\mathbb{E}[b_3]\right)y^{i}+\sum_{j=1}^{n}(f_{2}^{ij}+\mathbb{E}[f_{4}^{ij}])y^{j}+\sum_{j=1}^{n}(b_{2}^{j}+\mathbb{E}[b_{4}^{j}])\tilde{Y}^{j}_ry^i.
        \end{aligned}
        \end{equation*}
        Noting that $\tilde{Y}$ is uniformly bounded on $[T-\delta,T]$, one can easily check that $H$ and $\underline{H}$ are Lipschitz functions on $[T-\delta,T]$. Moreover, $\tilde{Y}$ is the unique solution of the following ODE on $[T-\delta, T]$,
        \begin{equation*}
            \tilde{Y}_u=h_{1}+\mathbb{E}[h_{2}]+\int_u^TH(r,\tilde{Y}_r)dr,
        \end{equation*}
        and the following ODE
        \begin{equation*}
             \underline{\mathbf{y}}_{u} = \int_{u}^{T}\underline{H}(r,\underline{\mathbf{y}}_{r})dr
            \end{equation*}
        admits a unique solution on $[T-\delta,T]$, satisfying $\underline{\mathbf{y}}_u=0$ for all $u\in[T-\delta,T]$. Recalling assumption (A3)(i), it follows from standard comparison theorem for ODEs that for all $u\in[T-\delta,T]$
        \begin{equation}\label{eq:lowerbound}
            0=\underline{\mathbf{y}}_u\leq \tilde{Y}_u.
        \end{equation}
        Similarly, for each $i=1,2,\cdots,n$ and $r\in[T-\delta,T],~y\in\mathbb{R}^n$, we introduce the following two functions:
        \begin{equation*}
            \begin{aligned}
             \tilde{H}^{i}(r,y) &= \left(f_1^{i}+\mathbb{E}[f_3^{i}]\right)+\left(b_1+\mathbb{E}[b_3]\right)y^{i}+\left(f_2^{i}+\mathbb{E}[f_4^{i}]\right)y+\left(b_2+\mathbb{E}[b_4]\right) \tilde{Y}_r\tilde{Y}^{i}_r \\
             &=\left(f_1^{i}+\mathbb{E}[f_3^{i}]\right)+\left(b_1+\mathbb{E}[b_3]\right)y^{i}+\sum_{j=1}^{n}(f_{2}^{ij}+\mathbb{E}[f_{4}^{ij}])y^{j}+\sum_{j=1}^{n}(b_{2}^{j}+\mathbb{E}[b_{4}^{j}])\tilde{Y}^{j}_r\tilde{Y}^i_r,\\
             \overline{H}^{i}(r,y)&=\left(f_1^{i}+\mathbb{E}[f_3^{i}]\right)+\left(b_1+\mathbb{E}[b_3]\right)y^{i}+\sum_{j=1}^{n}(f_{2}^{ij}+\mathbb{E}[f_{4}^{ij}])y^{j}.
            \end{aligned}
            \end{equation*}
            Again, one can easily check that $\tilde{H}$ and $\overline{H}$ are Lipschitz functions on $[T-\delta,T]$. Moreover, $\tilde{Y}$ is the unique solution of the following ODE on $[T-\delta, T]$,
            \begin{equation*}
                \tilde{Y}_u=h_{1}+\mathbb{E}[h_{2}]+\int_u^T\tilde{H}(r,\tilde{Y}_r)dr,
            \end{equation*}
            and the following ODE
            \begin{equation*}
            \overline{\mathbf{y}}_{u} = 2\mathbb{K}+\int_{u}^{T}\overline{H}(r,\overline{\mathbf{y}}_{r})dr
                \end{equation*}
            admits a unique solution on $[T-\delta,T]$, where $\mathbb{K}\in\mathbb{R}^n$ with all elements being $K$. Recalling assumption (A3)(i) and noting that $\tilde{Y}$ is positive on $[T-\delta,T]$, it follows from standard comparison theorem for ODEs that for all $u\in[T-\delta,T]$
            \begin{equation}\label{eq:upperbound}
                \tilde{Y}_u\leq \overline{\mathbf{y}}_u.
            \end{equation}
    Moreover, it holds from Gronwall's inequality that $|\overline{\mathbf{y}}_{u}| \leq M $ for all $u\in[T-\delta,T]$. Therefore, we have for all $u\in[T-\delta,T]$
    \begin{equation}
         |\mathbb{E}[\nabla Y_{u}]\mathbb{E}[\nabla X_{u}]^{-1}|\leq |\overline{\mathbf{y}}_{u}|\leq M. \label{multidimensional C}
    \end{equation}
    Now we consider the case when $\mathbb{E}[X_{T-\delta}^{t,\eta}]= \mathbb{E}[X_{T-\delta}^{t,\bar{\eta}}]$. Indeed, by using \eqref{eq:smallLip}, one can easily check that for $u\in[T-\delta,T]$, 
    \begin{equation*}
     \mathbb{E}[X_{u}^{t,\eta}]-\mathbb{E}[X_{u}^{t,\bar{\eta}}]=0, \quad   \mathbb{E}[Y_{u}^{t,\eta}]-\mathbb{E}[Y_{u}^{t,\bar{\eta}}]=0,
    \end{equation*}
    which together with \eqref{multidimensional C} yields
    \begin{equation}
        |\mathbb{E}[Y_{u}^{t,\eta}]-\mathbb{E}[Y_{u}^{t,\bar{\eta}}]|^{2} \leq M^2|\mathbb{E}[X_{u}^{t,\eta}]-\mathbb{E}[X_{u}^{t,\bar{\eta}}]|^2. \label{small duration expection 3}
    \end{equation}
    Now if $T-2\delta\leq s<T-\delta$, we can repeat the procedure in Step 1 with terminal time $T-\delta$ to obtain for any $t\in[s,T-\delta]$, it holds that
    \begin{equation*}
        \sup_{t\leq r\leq T-\delta}\left|\mathbb{E}[X^{t,\eta}_r]-\mathbb{E}[X^{t,\bar{\eta}}_r]\right|^2+\sup_{t\leq r\leq T-\delta}\left|\mathbb{E}[Y^{t,\eta}_{r}]-\mathbb{E}[Y^{t,\bar{\eta}}_{r}]\right|^2 \leq 2\left(M^2+1\right)|\mathbb{E}[\eta]-\mathbb{E}[\bar{\eta}]|^2,
    \end{equation*}
    which combining with \eqref{eq:estimateDelta} implies that for any $t\in[s,T]$, it holds that 
    \begin{equation*}
        \sup_{t\leq r\leq T}\left|\mathbb{E}[X^{t,\eta}_r]-\mathbb{E}[X^{t,\bar{\eta}}_r]\right|^2+\sup_{t\leq r\leq T}\left|\mathbb{E}[Y^{t,\eta}_{r}]-\mathbb{E}[Y^{t,\bar{\eta}}_{r}]\right|^2 \leq \left(2\left(M^2+1\right)\right)^2|\mathbb{E}[\eta]-\mathbb{E}[\bar{\eta}]|^2.
    \end{equation*}
    \textbf{Step 3:} Let $m$ be the smallest integer satisfying $T-m\delta\leq s$. Repeating the procedure in Step 2 finitely many times, we can get for any $t\in[s,T]$, 
    \begin{equation}
    \sup _{t\leq r \leq T}\left|\mathbb{E}[X_r^{t,\eta}]-\mathbb{E}[X_{r}^{t,\bar{\eta}}]\right|^2+ \sup _{t\leq r \leq T}\left|\mathbb{E}[Y_r^{t,\eta}]-\mathbb{E}[Y_{r}^{t,\bar{\eta}}]\right|^2\leq \left(2\left(M^2+1\right)\right)^m\mathbb{E}[\eta-\bar{\eta}]^{2}.
    \end{equation}
\end{proof}

Noting from Lemma \ref{small bound}, we have that for any $t\in[s,T]$ and $\eta\in L^{2}(\Omega,\mathcal{F}_{t},\mathbb{P};\mathbb{R})$, $\mathbb{E}[X^{t,\eta}_u]$ and $\mathbb{E}[Y^{t,\eta}_u]$ only depends on $\mathbb{E}[\eta]$, for all $u\in [t,T]$. Thus, $(X^{t,x,\eta},Y^{t,x,\eta},Z^{t,x,\eta})$ will only depend on the expectation of $\eta$. We are now ready to give the  definition of decoupling field $U$ and its uniformly Lipschitz continuity property, which is important for us to obtain global solvability for \eqref{mean field FBSDE}.

\begin{lemma}
Suppose assumptions (H) and (A1)-(A3) hold, for any $t\in[s,T]$ and $\eta \in L^{2}(\Omega,\mathcal{F}_{t},\mathbb{P};\mathbb{R})$, FBSDE \eqref{fbsde} has a unique solution $(X^{t,x,\eta},Y^{t,x,\eta},Z^{t,x,\eta}) \in \mathbb{S}^{2}(\mathbb{R}) \times \mathbb{S}^{2}(\mathbb{R}^{n})\times \mathbb{H}^{2}(\mathbb{R}^{n\times d})$ on $[t,T]$ and 
\begin{equation*}
     U(t,x,\mathbb{E}[\eta]) := Y_{t}^{t,x,\eta}
\end{equation*}
 defines a function on  $ \Omega \times [s,T] \times \mathbb{R}\times \mathbb{R}$, which satisfies
    \begin{equation}\label{pasting}
        Y_{r}^{t,\eta} = U(r,X_{r}^{t,\eta},\mathbb{E}[X_{r}^{t,\eta}]),\quad r\in[t,T].
    \end{equation}
 Further it satisfies that for any $t\in[s,T],x_1,x_2,\nu_{1},\nu_{2}\in\mathbb{R}$,
\begin{equation}
    |U(t,x_{1},\nu_{1})-U(t,x_{2},\nu_{2})|\leq C\left( |x_{1}-x_{2}|+\left|\nu_{1}-\nu_{2}\right|\right),\label{tildeU}
\end{equation}
where $C$ is a constant only depending on $n,K,T$ and independent of $t$.\label{small duration decoupling}
\end{lemma}

\begin{proof}
    For any $t\in[s,T]$ and $\eta \in L^{2}(\Omega,\mathcal{F}_{t},\mathbb{P};\mathbb{R})$, the solvability of FBSDE \eqref{fbsde} on $[t,T]$ follows
    as a straightforward consequence of our previous work \cite{hua2022unified}. Moreover, for any $x_1,x_2\in\mathbb{R},t\in[s,T]$, it holds that 
    \begin{equation}\label{eq:Ulipx}
        |Y^{t,x_1,\eta}_t-Y^{t,x_2,\eta}_t|\leq \bar{C}|x_1-x_2|,
    \end{equation}
    where $\bar{C}$ only depends on $n,K,T$ and is independent of $t$. The proof of decoupling field property is similar with \cite[Proposition 2.2]{chassagneux2022probabilistic} and will be omitted.
For any $x_2\in\mathbb{R},t\in[s,T]$ and $\eta_1,\eta_2 \in L^{2}(\Omega,\mathcal{F}_{t},\mathbb{P};\mathbb{R})$, a direct combination of Lemma \ref{lemma} and Lemma \ref{small bound} implies that 
\begin{equation}
    |Y_{t}^{t,x_2,\eta_{1}}-Y_{t}^{t,x_2,\eta_{2}}|\leq \tilde{C}|\mathbb{E}[\eta_1]-\mathbb{E}[\eta_{2}]|,\label{mean field decoupling}
    \end{equation}
    where $\tilde{C}$ only depends on $n,K,T$ and is independent of $t$. Now, for any $\nu_1,\nu_2\in\mathbb{R}$, taking $\eta_1,\eta_2 \in L^{2}(\Omega,\mathcal{F}_{t},\mathbb{P};\mathbb{R})$ which satisfy $\mathbb{E}[\eta_1]=\nu_1,\mathbb{E}[\eta_2]=\nu_2$, combining with \eqref{eq:Ulipx}, we obtain that 
    \begin{equation}
        \begin{aligned}
        |U(t,x_{1},\nu_{1})-U(t,x_{2},\nu_{2})| &\leq |U(t,x_{1},\nu_{1})-U(t,x_{2},\nu_{1})|+|U(t,x_{2},\nu_{1})-U(t,x_{2},\nu_{2})|\\
        &=|Y^{t,x_1,\eta_1}_t-Y^{t,x_2,\eta_1}_t|+|Y^{t,x_2,\eta_1}_t-Y^{t,x_2,\eta_2}_t|\\
        &\leq C\left(|x_{1}-x_{2}|+|\nu_{1}-\nu_{2}|\right),
        \end{aligned}\label{expection uniform estimate}
        \end{equation}
    where $C$ only depends on $n,K,T$ and is independent of $t$. The proof is complete.
\end{proof}
\begin{remark}
  If we extend our setup by assuming that $\left(\Omega, \mathcal{F}_0, \mathbb{P}\right)$ is rich enough to carry real valued random variables with any arbitrary distribution in $\mathcal{P}_2\left(\mathbb{R}\right)$, we can further represent $U$ function as some function $\tilde{U}$ defined on $\Omega \times [t, T] \times \mathbb{R} \times \mathcal{P}_2\left(\mathbb{R}\right) \mapsto \mathbb{R}^{n}$. Moreover, for any 
    $\mu\in \mathcal{P}_{2}(\mathbb{R})$, we can find $\eta\in L^{2}(\Omega,\mathcal{F}_{t},\mathbb{P};\mathbb{R})$ such that $\mathcal{L}(\eta) = \mu$, and according to the representation of Wassertein distance in terms of random variables (see \cite{carmona2018probabilistic}),
   \begin{equation*}
       W^2_{2}(\mu,\bar{\mu}) = \inf\left\{ \mathbb{E}[(\eta-\bar{\eta})^{2}];~\mathcal{L}(\eta)=\mu,\mathcal{L}(\bar{\eta})=\bar{\mu} \right\},
   \end{equation*}
    we get from \eqref{tildeU} that $\tilde{U}$ is uniformly Lipschitz, i.e.,
    \begin{equation*}
        |\tilde{U}(t,x_{1},\mu)-\tilde{U}(t,x_{2},\bar{\mu})|^{2} \leq \bar{C}\left(|x_{1}-x_{2}|^{2}+W^2_{2}(\mu,\bar{\mu})\right).
       \end{equation*}
\end{remark}

   We are now ready to give the proof of Theorem \ref{global existence}.

\begin{proof}[Proof of Theorem \ref{global existence}]
First, it follows from \cite[Theorem 4.24]{carmona2018probabilistic} that there exists a constant $\delta$ only depending on $C$ given in Lemma \ref{small duration decoupling} (noting that $C\geq K$) such that assumption (H) holds for $s=T-\delta$. Now we divide $[0,T]$ into $m$ small time intervals $0=t_{0}<$ $\cdots<t_{m}=T$ such that $t_{i}-t_{i-1} \leq \delta$, $i=1, \cdots, m$. Since assumption (H) holds for $s=T-\delta$, it follows from Lemma \ref{small duration decoupling} that the function $U(t,\cdot,\cdot)$ satisfies \eqref{tildeU} for any $t\in[t_{m-1},t_{m}]$, which in particular, implies that the Lipschitz constant of $ U(t_{m-1},\cdot,\cdot)$ is bounded by $C$. Thus with terminal condition $U\left(t_{m-1},\cdot,\cdot\right)$ and any initial condition $\eta \in L^{2}(\Omega,\mathcal{F}_{t_{m-2}},\mathbb{P};\mathbb{R})$, FBSDE \eqref{mean field fbsde} admits a unique solution on $\left[t_{m-2}, t_{m-1}\right]$. A standard pasting technique and uniqueness on each interval yield that assumption (H) holds for $s=t_{m-2}$. Repeating the above procedure finitely many times, we obtain that assumption (H) holds for $s=0$, which completes the proof.
\end{proof}
\begin{remark}\label{rem:del}
    In summary, under assumptions $(A1)-(A3)$, for any $t\in[0,T]$ and $\eta \in L^{2}(\Omega,\mathcal{F}_{t},\mathbb{P};\mathbb{R})$, FBSDE \eqref{mean field fbsde} has a unique solution $(X^{t,\eta},Y^{t,\eta},Z^{t,\eta}) \in \mathbb{S}^{2}(\mathbb{R}) \times \mathbb{S}^{2}(\mathbb{R}^{n})\times \mathbb{H}^{2}(\mathbb{R}^{n\times d})$, and it holds that for any $\eta, \bar{\eta} \in L^{2}(\Omega,\mathcal{F}_{t},\mathbb{P};\mathbb{R})$,
\begin{equation*}
    \mathbb{E}\left[|Y^{t,\eta}_t-Y^{t,\bar{\eta}}_t|^2\right]\leq C\mathbb{E}\left[|\eta-\bar{\eta}|^2\right]
\end{equation*}
where $C$ is a positive constant which does not depend on $\eta,\bar{\eta}$ and $t$. Thus, we provide a new situation where the general assumption (H3) in \cite{chassagneux2022probabilistic} is ensured.
\end{remark}
\subsection{Extension for one-dimensional case}
In this subsection, we provide an extension for the solvability of mean field FBSDE \eqref{mean field fbsde} when $n=1$. Indeed, the assumption (A2) (resp. (A3)) can be relaxed to the assumption (B1) (resp. (B2)).
The idea is similar to the multi-dimensional case, except a modification when getting the uniform bound of $\mathbb{E}[\nabla Y]\mathbb{E}[\nabla X]^{-1}$.
 First, for convenience, let us introduce the following two functions:
 \begin{equation*}
   F^{0}(t,y):=f_{1}(t)+(b_{1}(t)+f_{2}(t))y+b_{2}(t)y^{2}  \end{equation*}
 and 
 \begin{equation*}
 \begin{aligned}
F(\theta_{1},\theta_{2};t,y):=&f_{1}(t)+\mathbb{E}[f_{3}(t,\theta_{1},\theta_{2})]+\left(b_{1}(t)+\mathbb{E}[b_{3}(t,\theta_{1},\theta_{2})]\right.\\
 &\left.+f_{2}(t)+\mathbb{E}[f_{4}(t,\theta_{1},\theta_{2})]\right)y+(b_{2}+\mathbb{E}[b_{4}(t,\theta_{1},\theta_{2})])y^{2},    
 \end{aligned}
 \end{equation*}
 where $b_{3}(t,\theta_{1},\theta_{2}),b_{4}(t,\theta_{1},\theta_{2}),f_{3}(t,\theta_{1},\theta_{2}),f_{4}(t,\theta_{1},\theta_{2})$ are bounded measurable processes defined by \eqref{quadratic notation}. We further introduce the following assumptions when $n=1$. We denote 
 \begin{equation*}
    \varepsilon=\frac{1}{\left(K+(K+1)T\right)^2e^{4KT}},\quad \quad \tilde{\varepsilon}=\frac{1}{\left(2K+(2K+1)T\right)^2e^{8KT}}.
 \end{equation*}
\begin{itemize}
\item[\textbf{(B1)}] One of the following holds:\\
 (i) There exists a constant $\lambda$ such that for any $t\in[0,T]$
 \begin{equation}
    \lambda \leq h_{1},\quad F^{0}(t,\lambda) \geq 0, \quad  b_2(t) \leq \varepsilon.
 \end{equation}
(ii) There exists a constant $\lambda$ such that for any $t\in[0,T]$
\begin{equation}
    \lambda \geq h_{1},\quad F^{0}(t,\lambda) \leq 0, \quad b_2(t) \geq-\varepsilon.
\end{equation}
\item[\textbf{(B2)}] One of the following holds:\\
 (i) There exists a constant $\lambda$ such that for any $t\in[0,T]$ and any $\theta_{1}=(\bar{x}_{1},\bar{y}_{1}),\theta_{2}=(\bar{x}_{2},\bar{y}_{2})\in\mathbb{R}\times \mathbb{R}$,
 \begin{equation}
    \lambda \leq h_{1}+{h}_{2}(\bar{x}_{1},\bar{x}_{2}),\quad F(\theta_{1},\theta_{2};t,\lambda) \geq 0, \quad     b_2(t)+b_{4}(t,\theta_{1},\theta_{2})\leq \tilde{\varepsilon}.
 \end{equation}
(ii) There exists a constant $\lambda $ such that for any $t\in[0,T]$ and any $\theta_{1}=(\bar{x}_{1},\bar{y}_{1}),\theta_{2}=(\bar{x}_{2},\bar{y}_{2})\in\mathbb{R}\times \mathbb{R}$,
\begin{equation}
    \lambda \geq h_{1}+{h}_{2}(\bar{x}_{1},\bar{x}_{2}),\quad F(\theta_{1},\theta_{2};t,\lambda)  \leq 0, \quad b_2(t)+b_{4}(t,\theta_{1},\theta_{2}) \geq-\tilde{\varepsilon}.   
  \end{equation}
\end{itemize}
\begin{remark}
If $h_{1} \geq 0$, $f_{1}(t)\geq 0$ and $b_2(t)\leq 0$ for any $t\in[0,T]$, then (B1)(i) is satisfied for $\lambda=0$.
Sufficient conditions can be similarly given for assumptions (B1)(ii) and (B2).
 \label{remark}
\end{remark}
Now we give the main result of this subsection, which is an extension of Theorem \ref{global existence}.
\begin{Theorem}
Suppose $n=1$, under assumption $(A1)$ and assumptions $(B1)-(B2)$, the mean field FBSDE \eqref{mean field FBSDE} has a unique solution $(X,Y,Z) \in \mathbb{S}^{2}(\mathbb{R}) \times \mathbb{S}^{2}(\mathbb{R})\times \mathbb{H}^{2}(\mathbb{R}^{d})$.
\end{Theorem}
\begin{proof}
According to \cite{ma2015well}, the argument of Lemma \ref{lemma} still holds under assumptions (A1) and (B1). We will now show the argument of Lemma \ref{small bound} still holds true under assumptions (A1) and (B2). In particular, Let $\overline{M}:=\max\{|\lambda|,\left[2K+(2K+1)T\right]e^{4KT}\}$, where $\lambda$ is the constant in assumption (B2), and $\overline{\delta}=\frac{1}{2\overline{M}(3\overline{M}^2+4)}$. Noting that $\overline{M}\geq 2K+(2K+1)Te^{4KT}$ with $n=1$, the results of step 1 in Lemma \ref{small bound} also hold if $s\geq T-\overline{\delta}$. We are only left to get the uniform bound of $\mathbb{E}[\nabla Y]\mathbb{E}[\nabla X]^{-1}$. Without loss of any generality, we only give the proof under assumption (B2)(i) and we will omit the same part and only focus on the different part. With same notations as in Lemma \ref{small bound}, it holds that on $[T-\bar{\delta},T]$,
\begin{equation}
\begin{aligned}
    d\tilde{Y}_{r}&=-\left[(f_{1}+\mathbb{E}[f_{3}])+(b_{1}+\mathbb{E}[b_{3}]+f_{2}+\mathbb{E}[f_{4}])\tilde{Y}_{r}+(b_{2}+\mathbb{E}[b_{4}])\tilde{Y}_{r}^{2}\right]dr\\
    &=-F(\Theta_{1},\Theta_{2}; r,\tilde{Y}_r)dr.
\end{aligned}\label{ito formula}
\end{equation}
We now consider the following ODE 
\begin{equation*}
\bar{\mathbf{y}}_{u} = 2K+\int_{u}^{T}\left[4K\bar{\mathbf{y}}_{r}+2K+1\right]dr,
\end{equation*}
which, following from standard ODE theory, admits a unique solution. Moreover, it follows from Gronwall's inequality that  
\begin{equation*}
\bar{\mathbf{y}}_{u}\leq \tilde{M}:=\left[2K+(2K+1)T\right]e^{4KT}.
\end{equation*}
Noting that under assumptions (A1) and (B2)(i), we have $\tilde{\varepsilon} \bar{\mathbf{y}}_{r}^{2}\leq 1$. Therefore, the above ODE has the following equivalent form 
\begin{equation}\label{bar y}
    \bar{\mathbf{y}}_{u} = 2K+\int_{u}^{T}\left[\tilde{\varepsilon}\bar{\mathbf{y}}^{2}_{r}+4K\bar{\mathbf{y}}_{r}+2K+(1-\tilde{\varepsilon} \bar{\mathbf{y}}^{2}_{r})\right]dr.
\end{equation}
Moreover, it holds that\footnote{It is obvious that $\bar{\mathbf{y}}$ is non-negative.}
\begin{equation}\label{upper compare}
F(\Theta_{1},\Theta_{2};r,\bar{\mathbf{y}}_{r}) \leq \tilde{ \varepsilon} {\bar{\mathbf{y}}_{r}}^{2}+4K\bar{\mathbf{y}}_{r}+2K.
\end{equation}
On the other hand, the following ODE
\begin{equation*}
\underline{\mathbf{y}}_{u} = \lambda+\int_{u}^{T}F(\Theta_{1},\Theta_{2};r,\underline{\mathbf{y}}_{r})-F(\Theta_{1},\Theta_{2};r,\lambda)dr,
\end{equation*}
admits a unique solution and $\underline{\mathbf{y}}\equiv \lambda$.
Under assumption (B2)(i), we have $\lambda\leq h_{1}+\mathbb{E}[h_{2}]$ and $F(\Theta_{1},\Theta_{2};r,\lambda)\geq 0$. Combining with \eqref{upper compare} and \eqref{bar y}, applying Lemma 5.1 in \cite{ma2015well}, we obtain 
\begin{equation*}
\lambda=\underline{\mathbf{y}}\leq \tilde{Y}\leq \bar{\mathbf{y}}\leq \tilde{M}.
\end{equation*}  
Thus, following similar arguments as Lemma \ref{small bound}, we can get for any $u\in[T-\overline{\delta},T]$,
\begin{equation}
\left|\mathbb{E}\left[Y_u^{t, \eta}\right]-\mathbb{E}\left[Y_u^{t,\bar{\eta}}\right]\right|^2 \leq \bar{M}^2\left|\mathbb{E}\left[X_u^{t, \eta}\right]-\mathbb{E}\left[X_u^{t,\bar{\eta}}\right]\right|^2.\label{1-dimensional bound}
\end{equation}
 Under assumption (B2)(ii),  \eqref{1-dimensional bound} can be proved similarly. Finally, the statement follows from similar arguments used in Lemma \ref{small bound}, Lemma \ref{small duration decoupling} and Theorem \ref{global existence}. The proof is complete.
\end{proof}
\subsection{Comparison to the existing results}\label{sec:com}
\indent In this subsection, we apply our results to investigate some examples to compare with some of the existing results. In the first example, we revisit a mean-field type FBSDE from \cite{bensoussan2015well}. Our analysis shows that our method would provide a simpler way to verify the solvability. In the second example, we further consider some mean field type FBSDEs which can be handled by our approch, while the conditions needed in \cite{bensoussan2015well} will not be satisfied. Finally, in the last example, we study a mean field type FBSDE arising from a mean field LQ control problem (see \cite{tian2023mean}), whose solvability is given by our approach under conditions beyond the domination-monotonicity condition required in \cite{tian2023mean}. \\
\textbf{Example 1.} For $\alpha \in \mathbb{R}, \lambda \in[0,1]$ and $\theta \in[0,1]$, we consider the following scalar mean field type FBSDE \footnote{see \cite{carmona2013mean} for the special case where $\alpha=1, \lambda=1$ and $\theta=1$. }, 
\begin{equation}
\left\{\begin{array}{l}
d X_t=\alpha \mathbb{E}\left[Y_t\right] d t, \\
d Y_t=-\left(\lambda \mathbb{E}\left[X_t\right]+(1-\lambda) X_t\right) d t+Z_t d W_t, \\
X_0=0, \quad Y_T=\theta \mathbb{E}\left[X_T\right]+(1-\theta) X_T,
\end{array}\right.\label{example 1}
\end{equation}
where $T \in \mathbb{R}^{+}$and satisfies
\begin{equation}
\alpha \sin (\sqrt{\alpha} T)=\sqrt{\alpha} \cos (\sqrt{\alpha} T), \label{alpha condition}
\end{equation}
when $\alpha \geq 0$.\\
\indent We study this equation under $\alpha >0$ and $\alpha  \leq 0$.\\
(a) If $\alpha >0$, since $\lambda \in[0,1]$ and $\theta \in[0,1]$, then we have $f_{1}\ge 0$, $b_2>0$, $f_{1}+f_{3}=1\ge 0$, $h_{1}\ge 0$ and $h_{1}+h_{2}\ge 0$. Thus assumptions (A2) and (A3) are not satisfied. Therefore there might possess non-unique solutions. Indeed, it has been shown in \cite{bensoussan2015well} that the following expressions could serve as solutions for this system:
\begin{equation*}
X_t=K_1 \sin \sqrt{\alpha} t, \quad Y_t=\frac{K_1}{\sqrt{\alpha}} \cos \sqrt{\alpha} t, \quad Z_t=0, \quad K_1 \in \mathbb{R} .
\end{equation*}
(b) If $\alpha \leq 0$, we can easily verify that the system \eqref{example 1} satisfies the assumptions (A1)-(A3), then it follows from Theorem \ref{global existence} that there exists a unique solution for the system, which is indeed given by 
\begin{equation*}
    X_{t} =Y_{t} = Z_{t} = 0.
\end{equation*}
\noindent\textbf{Example 2.} For $ \alpha,\beta \in \mathbb{R}, \lambda \in[0,1], $ and $\theta \in[0,1]$, let us consider the following scalar mean field type FBSDE, 
\begin{equation}
\left\{\begin{array}{l}
d X_t=\left(\alpha \mathbb{E}\left[Y_t\right] + \beta \mathbb{E}\left[X_{t}\right]\right) d t,\\
d Y_t=-\left(\lambda \mathbb{E}\left[X_t\right]+(1-\lambda) X_t\right) d t+Z_t d W_t, \\
X_0=0, \quad Y_T=\theta \mathbb{E}\left[X_T\right]+(1-\theta) X_T.
\end{array}\right.\label{example 2}
\end{equation}  
\indent With our approach, except the Lipschitz continuity, no additional assumption on the forward coefficient with respect to $\mathbb{E}[X_{t}]$ is needed. Thus the additional term $\beta \mathbb{E}[X_{t}]$ has no effect on the conclusion in the example 1, which means we can still get existence and uniqueness of solution when $\alpha\leq 0$. However, this might not be the case when using the continuation method. Let use denote
\begin{equation*}
\begin{aligned}
& U_t=\left(\begin{array}{c}
X_t \\
Y_t \\
Z_t
\end{array}\right), \quad A\left(t ; U_t, \mathbb{P}_{\left(X_t, Y_t, Z_t\right)}\right)=\left(\begin{array}{c}
-\lambda \mathbb{E}\left[X_t\right]-(1-\lambda) X_t \\
\alpha \mathbb{E}\left[Y_t\right] +\beta \mathbb{E}\left[ X_{t}\right]\\
0
\end{array}\right), \\
& g\left(X, \mathbb{P}_{X_T}\right)=\theta \mathbb{E}\left[X_T\right]+(1-\theta) X_T .
\end{aligned}
\end{equation*}
Then we have 
\begin{equation*}
    \begin{aligned}
    \mathbb{E} & {\left[\left\langle g\left(X_T^1, \mathbb{P}_{X_T^1}\right)-g\left(X_T^2, \mathbb{P}_{X_T^2}\right), X_T^1-X_T^2\right\rangle\right] } \\
    & =\theta\left(\mathbb{E}\left[X_T^1-X_T^2\right]\right)^2+(1-\theta) \mathbb{E}\left[\left(X_T^1-X_T^2\right)^2\right] ,
    \end{aligned}
    \end{equation*}
\begin{equation*}
\begin{aligned}
\mathbb{E} & {\left[\left(A\left(t ; U_t^1, \mathbb{P}_{\left(X_t^1, Y_t^1, Z_t^1\right)}\right)-A\left(t ; U_t^2, \mathbb{P}_{\left(X_t^2, Y_t^2, Z_t^2\right)}\right), U_t^1-U_t^2\right)\right] } \\
& =-\lambda\left(\mathbb{E}\left[X_t^1-X_t^2\right]\right)^2-(1-\lambda) \mathbb{E}\left[\left(X_t^1-X_t^2\right)^2\right]+\alpha\left(\mathbb{E}\left[Y_t^1-Y_t^2\right]\right)^2+\alpha \beta \mathbb{E}\left[ Y_{t}^{1}-Y_{t}^{2}\right]\mathbb{E}\left[ X_{t}^{1}-X_{t}^{2}\right].
\end{aligned}
\end{equation*}
Thus assumption (A4) in \cite{bensoussan2015well} might not be satisfied in general.\\
\textbf{Example 3. } Let us consider the following mean field type FBSDE,
\begin{equation}
\left\{\begin{aligned}
dX_{t} & = \left\{A_{t}X_{t}+\bar{A}_{t} \mathbb{E}\left[X_{t}\right]-\frac{B_{t}^{2}}{R_{t}}(Y_{t}-\mathbb{E}[Y_{t}]) -\frac{(B_{t}+\bar{B}_{t})^2}{R_{t}+\bar{R}_{t}} \mathbb{E}[Y_{t}]\right\} dt+\sigma_{t} dW_{t}, \quad t \in[0, T], \\
dY_{t} & =  -\left\{A_{t}Y_{t}+\bar{A}_{t} \mathbb{E}[Y_{t}]+Q_{t}X+\bar{Q}_{t} \mathbb{E}\left[X_{t}\right]\right\} dt + Z_{t}dW_t, \quad t \in[0, T], \\
X_{0}& = x, \quad Y_{T}=G X_{T}+\bar{G} \mathbb{E}\left[X_{T}\right].
\end{aligned}\right.\label{lq example}
\end{equation}\\
\indent We further introduce the following conditions.
\begin{itemize}
\item[(\bf{E1})] $A(\cdot),\bar{A}(\cdot),B(\cdot),\bar{B}(\cdot),Q(\cdot),\bar{Q}(\cdot),R(\cdot),\bar{R}(\cdot)$ are bounded deterministic processes, $G,\bar{G}$ are constants.
    \item [(\bf{E2})] One of the following cases holds:\\
    (i) $G \geq 0, R(\cdot) > 0, Q(\cdot) \geq 0$;\\
    (ii) $G \leq 0, R(\cdot)< 0, Q(\cdot) \leq 0$;
     \item [(\bf{E3})] One of the following cases holds:\\
    (i) $G+\bar{G} \geq 0, R(\cdot)+\bar{R}(\cdot)  > 0, Q(\cdot)+\bar{Q}(\cdot)  \geq 0$;\\
    (ii) $G+\bar{G} \leq 0, R(\cdot)+\bar{R}(\cdot) < 0,Q(\cdot)+\bar{Q}(\cdot)  \leq 0$;
\end{itemize}

\indent Under conditions (E1)-(E3), it is easy to check that the coefficients of mean field type FBSDE \eqref{lq example} satisfy assumptions (A1)-(A3). Therefore it follows from Theorem \ref{global existence} that \eqref{lq example} admits a unique solution $(X,Y,Z)$. It should be mentioned that under conditions (E1), (E2)(i), (E3)(i) or conditions (E1), (E2)(ii), (E3)(ii), domination-monotonicity conditions required in \cite{tian2023mean} are satisfied, from which the solvability of \eqref{lq example} follows. However, one can easily verify that the domination-monotonicity conditions fail under conditions (E1), (E2)(i), (E3)(ii) or conditions (E1), (E2)(i), (E3)(ii).

\section{A representation result}

In this section, we will give a representation result for mean field FBSDE \eqref{mean field fbsde}. First, let us introduce the following Riccati equation
 \begin{equation}
     \left\{
     \begin{aligned}
         &dP_{t}=-(b_{2}(t) P_{t})P_{t}-f_{2}(t) P_{t}-b_{1}(t)P_{t}-f_{1}(t),\\
         &P_{T} = h_{1}.\label{p equation}
     \end{aligned}
     \right.
 \end{equation}
 \begin{lemma}
 Under assumptions $(A1)-(A3)$, the Riccati equation \eqref{p equation} admits a unique solution on $[0,T]$ such that for any $t\in[0,T]$,
 \begin{equation*}
    \left|P_{t}\right| \leq C,
 \end{equation*}
 where $C$ is a constant depending only on $n,K,T$.
 \label{P bound}
 \end{lemma} 
 \begin{proof}
 Under assumptions (A1)-(A3), we can show that the existence and uniqueness of solution of the equation \eqref{p equation} on [0,T] following similar argument as \cite{hua2022unified} and there exists a constant $C$ depending only on $n,K,T$ such that for any $t\in[0,T]$, $\left|P_{t}\right| \leq C$.
 \end{proof}
 
For any $t\in[0,T],\nu\in\mathbb{R}$, we now introduce the following mean field FBSDE
 \begin{equation}
 \left\{
     \begin{aligned}
      d \nu_{s}^{t,\nu} &= (b_{1}(s) +b_{2}(s) P_{s}) \nu_{s}^{t,\nu} + b_{2}(s)\mathbb{E}[\varphi_{s}^{t,\nu}] + \mathbb{E}[ b_{0}(s,\nu_{s},P_{s}\nu_{s}^{t,\nu}+\mathbb{E}[\varphi_{s}^{t,\nu}])]ds,\\
      d\varphi_{s}^{t,\nu}&=- f_{2}(s)\varphi_{s}^{t,\nu}-(b_{2}(s)\varphi_{s}^{t,\nu})P_{s} -f_{0}(s,\nu_{s}^{t,\nu},P_{s}\nu_{s}^{t,\nu}+\mathbb{E}[\varphi_{s}^{t,\nu}])\\
      &\quad-P_{s}b_{0}(s,\nu_{s}^{t,\nu},P_{s}\nu_{s}^{t,\nu}+\mathbb{E}[\varphi_{s}^{t,\nu}])ds+z^{t,\nu}_{s}dW_{s},\\
      \nu_{t}&=\nu,\quad \varphi_{T} = h_{2}(\nu_{T}^{t,\nu}).
     \end{aligned}
     \right.\label{v equation}
 \end{equation}
 \begin{Theorem}
 Let assumptions $(A1)-(A3)$ hold, then for any $(t,\nu) \in [0,T]\times \mathbb{R}$, FBSDE \eqref{v equation} admits a unique solution such that for any $s\in[t,T]$,
 \begin{equation}
   \nu^{t,\nu}_s= \mathbb{E}[X_{s}^{t,\eta}],\quad \varphi_{s}^{t,\nu}=Y_{s}^{t,\eta}-P_{s}X_{s}^{t,\eta},\label{u x v}
   \end{equation}
where $X^{t,\eta},Y^{t,\eta}$ are the first two components of the unique solution of mean field FBSDE \eqref{mean field fbsde} with initial condition $\eta\in L^{2}(\Omega,\mathcal{F}_{t},\mathbb{P};\mathbb{R})$ satisfying $\mathbb{E}[\eta] = \nu$ and $P$ is the unique solution of \eqref{p equation}.\label{phi theorem}
\end{Theorem}
\begin{proof} For any $(t,\nu)\in[0,T]\times\mathbb{R}$ and $\eta\in L^{2}(\Omega,\mathcal{F}_{t},\mathbb{P};\mathbb{R})$ satisfying $\mathbb{E}[\eta] = \nu$, it follows from Theorem \ref{global existence} that mean field FBSDE \eqref{mean field fbsde} admits a unique solution $(X^{t,\eta},Y^{t,\eta},Z^{t,\eta})$. Thus, it follows from standard theory of mean field BSDE theory (see \cite{carmona2018probabilistic}) that the following mean field BSDE
\begin{equation}
\left\{
    \begin{aligned}
      d\varphi_{s}&=- f_{2}(s)\varphi_{s}-(b_{2}(s)\varphi_{s})P_{s} -f_{0}(s,\mathbb{E}[X_{s}^{t,\eta}],P_{s}\mathbb{E}[X_{s}^{t,\eta}]+\mathbb{E}[\varphi_{s}])\\
       &\quad-P_{s}b_{0}(s,\mathbb{E}[X_{s}^{t,\eta}],P_{s}\mathbb{E}[X_{s}^{t,\eta}]+\mathbb{E}[\varphi_{s}])ds+z_{s}dW_{s},\\
      \varphi_{T} & = h_{2}(\mathbb{E}[X_{T}^{t,\eta}]),
    \end{aligned}\right.\label{mean field bsde}
\end{equation}
is well-posed on $[t,T]$. Now for any $s\in[t,T]$, applying It\^{o}'s formula to $P_{s}X_{s}^{t,\eta}+\varphi_{s}-Y^{t,\eta}_s$, it yields
\begin{equation*}
\begin{aligned}
d(P_{s}X_{s}^{t,\eta}+\varphi_{s}-Y^{t,\eta}_s) = &(-(b_{2}(s) P_{s})P_{s}-f_{2}(s) P_{s}-b_{1}(s)P_{s}-f_{1}(s))X_s^{t,\eta} ds  \\
& +P_s\left[b_{1}(s)X_s^{t,\eta}+b_{2}(s) Y_{s}^{t,\eta}+ b_{0}(s,\mathbb{E}[X^{t,\eta}_{s}],\mathbb{E}[Y_{s}^{t,\eta}])\right]ds\\
&\quad+P_s \sigma(s,X_{s}^{t,\eta},Y_{s}^{t,\eta},\mathbb{E}[X_{s}^{t,\eta}],\mathbb{E}[Y_{s}^{t,\eta}]) d W_s \\
& +\left[-f_{2}(s)\varphi_{s}-P_{s}(b_{2}(s)\varphi_{s}) -f_{0}(s,\mathbb{E}[X_{s}^{t,\eta}],P_{s}\mathbb{E}[X_{s}^{t,\eta}]+\mathbb{E}[\varphi_{s}])\right.\\
&\left.-P_{s}b_{0}(s,\mathbb{E}[X_{s}^{t,\eta}],P_{s}\mathbb{E}[X_{s}^{t,\eta}]+\mathbb{E}[\varphi_{s}])\right] d s+z_{s}dW_{s},\\
&+\left(f_{1}(s)X_s^{t,\eta}+f_{2}(s)Y_s^{t,\eta}+f_{0}(s,\mathbb{E}[X_s^{t,\eta}],\mathbb{E}[Y_{s}^{t,\eta}])\right)dt-Z^{t,\eta}_sdW_{s}.
\end{aligned}
\end{equation*}
After rearrangement, we have
\begin{equation}\label{BSDE 0}
    \begin{aligned}
    d(P_{s}X_{s}^{t,\eta}+\varphi_{s}-Y^{t,\eta}_s) = &\left(-f_2(s)\left(P_{s}X_{s}^{t,\eta}+\varphi_{s}-Y^{t,\eta}_s\right)-P_s\left(b_2(s)(P_{s}X_{s}^{t,\eta}+\varphi_{s}-Y^{t,\eta}_s)\right)\right)ds\\
    &+P_s\left(b_{0}(s,\mathbb{E}[X^{t,\eta}_{s}],\mathbb{E}[Y_{s}^{t,\eta}])-b_{0}(s,\mathbb{E}[X_{s}^{t,\eta}],P_{s}\mathbb{E}[X_{s}^{t,\eta}]+\mathbb{E}[\varphi_{s}])\right)ds\\
    &+\left(f_{0}(s,\mathbb{E}[X_s^{t,\eta}],\mathbb{E}[Y_{s}^{t,\eta}])-f_{0}(s,\mathbb{E}[X_{s}^{t,\eta}],P_{s}\mathbb{E}[X_{s}^{t,\eta}]+\mathbb{E}[\varphi_{s}])\right)ds\\
    & +\left(P_s \sigma(s,X_{s}^{t,\eta},Y_{s}^{t,\eta},\mathbb{E}[X_{s}^{t,\eta}],\mathbb{E}[Y_{s}^{t,\eta}]) +z_s-Z^{t,\eta}_s\right)dW_{s}.
    \end{aligned}
    \end{equation}
Noting that $P_T X_T^{t,\eta}+\varphi_T-Y^{t,\eta}_T=h_{1}X_T^{t,\eta}+h_{2}\left(\mathbb{E}[X_{T}^{t,\eta}]\right)-h_{1}X_T^{t,\eta}-h_{2}\left(\mathbb{E}[X_{T}^{t,\eta}]\right)=0$, it follows from standard ODE and BSDE theory that for all $s\in[t,T]$,
\begin{equation}
\varphi_{s}=
  Y_{s}^{t,\eta}-P_{s}X_{s}^{t,\eta}. \label{linear form} 
\end{equation}
 \indent On the other hand, one can esasily verify the following ODE
 \begin{equation*}
    \left\{
     \begin{aligned}
      d \nu_{s} &= (b_{1}(s) +b_{2}(s) P_{s}) \nu_{s} + b_{2}(s)\mathbb{E}[Y_{s}^{t,\eta}-P_{s}X_{s}^{t,\eta}] + \mathbb{E}[ b_{0}(s,\nu_{s},P_{s}\nu_{s}+\mathbb{E}[Y_{s}^{t,\eta}-P_{s}X_{s}^{t,\eta}])]ds, \\
      \nu_t&=\nu,
     \end{aligned}
    \right.
 \end{equation*}
 admits a unique solution which satisfies $\nu_{s} = \mathbb{E}[X_{s}^{t,\eta}]$ for any $s\in[t,T]$. Indeed, we have
 \begin{equation}
 \begin{aligned}
  d \mathbb{E}[X_{s}^{t,\eta}]    &=b_{1}(s)\mathbb{E}[X_{s}^{t,\eta}]+b_{2}(s)\mathbb{E}[Y_{s}^{t,\eta}]+\mathbb{E}[b_{0}(s,\mathbb{E}[X_{s}^{t,\eta}],\mathbb{E}[Y_{s}^{t,\eta}])]ds\\
  &= (b_{1}(s) +b_{2}(s) P_{s}) \mathbb{E}[X_{s}^{t,\eta}] + b_{2}(s)\mathbb{E}[Y_{s}^{t,\eta}-P_{s}X_{s}^{t,\eta}] \\
  &\quad+ \mathbb{E}[ b_{0}(s,\mathbb{E}[X_{s}^{t,\eta}],P_{s}\mathbb{E}[X_{s}^{t,\eta}]+\mathbb{E}[Y_{s}^{t,\eta}-P_{s}X_{s}^{t,\eta}])]ds,
 \end{aligned}
 \end{equation}
 and $\mathbb{E}[X^{t,\eta}_{t}] = \mathbb{E}[\eta] = \nu$. Therefore, in view of the solvability of mean field BSDE \eqref{mean field bsde} and the relation \eqref{linear form}, FBSDE \eqref{v equation} admits a solution satifying \eqref{u x v}. Finally, the uniqueness follows from the standard local well-posedness theory of mean field FBSDEs (see e.g. \cite{carmona2018probabilistic}), relation \eqref{u x v} and Lemma \ref{small duration decoupling}.
 \end{proof}
\section{Classical solutions of related master equations}
In this section, we will consider classical solutions of master equations which are related to mean field FBSDE \eqref{mean field fbsde}. More precisely, we study classical solution of the following master equation
\begin{equation}
\left\{\begin{aligned}
&\partial_t U(t, x, \nu)+\partial_x U(t, x, \nu)\left[b_{1}(t)x+b_{2}(t) U(t, x, \nu)+b_{0}\left(t,\nu,\mathbb{E}[U(t, \eta, \nu)]\right)\right]\\
&+\frac{1}{2} \partial_{x x} U(t, x, \nu)\left|\sigma(t,x,U(t,x,\nu),\nu,\mathbb{E}[U(t,\eta,\nu)])\right|^2\\
&+\partial_\nu U(t, x, \nu)\left[b_{1}(t)\nu+b_{2}(t) \mathbb{E}[U(t, \eta, \nu)]+b_{0}(t,\nu,\mathbb{E}[U(t, \eta, \nu)])\right] \\
&+f_{1}(t)x+f_{2}(t)U(t,x,\nu)+f_{0}(t,\nu,\mathbb{E}[U(t,\eta,\nu)])=0, \\
&U(T, x, \nu)=h_{1}x+h_{2}(\nu),
\end{aligned}\right.\label{master equation}
\end{equation}
where $\eta \in L^{2}(\Omega,\mathcal{F}_{t},\mathbb{P};\mathbb{R})$ with $\mathbb{E}[\eta] = \nu$.  We further introduce the following assumption:\\
\textbf{Assumption (A4):} The functions $f_{0},b_{0},\sigma,h_{2}$ are deterministic satisfying for any $t\in[0,T]$,
\begin{equation*}
    |b_0(t,0,0)|\leq K,~|f_0(t,0,0)|\leq K,~|\sigma(t,0,0,0,0)|\leq K,~|h_2(0)|\leq K,
\end{equation*}
 and $f_{0}(t,\cdot,\cdot)\in C^{1}(\mathbb{R}\times \mathbb{R}^{n};\mathbb{R}^n),b_{0}(t,\cdot,\cdot)\in C^{1}(\mathbb{R}\times \mathbb{R}^{n};\mathbb{R})$ and $h_{2}\in C^{1}(\mathbb{R};\mathbb{R}^{n})$.\\
\indent In this setting, the mean field FBSDE \eqref{v equation} degenerates to the following forward backward ordinary differential equation (FBODE)
\begin{equation}
\left\{
    \begin{aligned}
      d \nu_{s}^{t,\nu} &= (b_{1}(s) +b_{2}(s) P_{s}) \nu_{s}^{t,\nu} + b_{2}(s)\varphi_{s}^{t,\nu} +  b_{0}(s,\nu_{s}^{t,\nu},P_{s}\nu_{s}^{t,\nu}+\varphi_{s}^{t,\nu})ds\\
      d\varphi_{s}^{t,\nu}&=- f_{2}(s)\varphi_{s}^{t,\nu}-(b_{2}(s)\varphi_{s}^{t,\nu})P_{s} -f_{0}(s,\nu_{s}^{t,\nu},P_{s}\nu_{s}^{t,\nu}+\varphi_{s})-P_{s}b_{0}(s,\nu_{s}^{t,\nu},P_{s}\nu_{s}^{t,\nu}+\varphi_{s}^{t,\nu})ds,\\
      \nu_{t}^{t,\nu}&=\nu,\varphi_{T} ^{t,\nu}= h_{2}(\nu_{T}^{t,\nu}).
     \end{aligned}
     \right.\label{markovian phi}
\end{equation}
\begin{Theorem}
    Let assumptions $(A1)-(A3)$ hold and $f_{0},b_{0},\sigma,h_{2}$ are deterministic, then for any $(t,\nu) \in [0,T]\times \mathbb{R}$, FBODE \eqref{markovian phi} admits a unique solution such that for any $s\in[t,T]$,
    \begin{equation*}
      \nu^{t,\nu}_s= \mathbb{E}[X_{s}^{t,\eta}],\quad \varphi_{s}^{t,\nu}=Y_{s}^{t,\eta}-P_{s}X_{s}^{t,\eta},
      \end{equation*}
   where $X^{t,\eta},Y^{t,\eta}$ are the first two components of the unique solution of mean field FBSDE \eqref{mean field fbsde} with initial condition $\eta\in L^{2}(\Omega,\mathcal{F}_{t},\mathbb{P};\mathbb{R})$ satisfying $\mathbb{E}[\eta] = \nu$ and $P$ is the unique solution of \eqref{p equation}.\label{markovian phi theorem}
   \end{Theorem}
   \begin{proof}
    The proof follows immediately by a slight modification of the proof of Theorem \ref{phi theorem}.
   \end{proof}
 
Now we define the funtion $\Phi:[0,T]\times\mathbb{R}\rightarrow \mathbb{R}^{n}$ as
\begin{equation}
    \Phi(t,\nu) = \varphi_{t}^{t,\nu},\label{phi decoupling}
\end{equation}
and further it holds that $\varphi_{s}^{t,\nu}=\Phi(s,\nu_{s}^{t,\nu})$, for all $s \in [t,T]$ (see \cite{delarue2002existence}), which implies that $\Phi(t,\nu)$ 
corresponds to the following PDE
\begin{equation}
\left\{
    \begin{aligned}
     &\partial_{t}\Phi(t,\nu)+\partial _{\nu}\Phi(t,\nu)\left[(b_{1}(t)+b_{2}(t)P_{t})\nu+b_{2}(t)\Phi(t,\nu)+b_{0}(t,\nu,P_{t}\nu+\Phi(t,\nu))\right]\\
     &+(b_{2}(t) P_{t})\Phi(t,\nu)+f_{2}(t)\Phi(t,\nu)+f_{0}(t,\nu,P_{t}\nu+\Phi(t,\nu))+P_{t}b_{0}(t,\nu,P_{t}\nu+\Phi(t,\nu))=0,\\
     &\Phi(T,\nu) = h_{2}(\nu).
    \end{aligned}\right.\label{phi pde}
\end{equation}
We would like to show that 
 $\Phi\in C^{1,2}([0,T]\times\mathbb{R};\mathbb{R}^{n})$ to verify $\Phi$ is indeed a classical solution to \eqref{phi pde}. Now let us consider the following FBODE on $[t,T]$, which can be interpreted as a formal differentiation of \eqref{markovian phi} with respect to initial condition $\nu$:
\begin{equation}
\left\{\begin{aligned}
d \nabla \nu_s^{t, \nu}= & \left[(b_{1}(s)+b_{2}(s) P_{s})\nabla\nu_{s}^{t,\nu}+b_{2}(s) \nabla\varphi_{s}^{t,\nu}\right.\\
&\left.+\partial_{\nu}b_{0}(s,\nu_{s}^{t,\nu},P_{s}\nu^{t,\nu}_{s}+\varphi_{s}^{t,\nu})\nabla\nu_{s}^{t,\nu}+\partial_{\varphi}b_{0}(s,\nu_{s}^{t,\nu},P_{s}\nu^{t,\nu}_{s}+\varphi_{s}^{t,\nu})\nabla\varphi_{s}^{t,\nu}\right]ds\\
d \nabla \varphi_s^{t, \nu}= & \left[- 
f_{2}(s)\nabla \varphi_{s}^{t,\nu} -P_{s}( b_{2}(s)\nabla \varphi_{s}^{t,\nu})-\partial_{\nu}f_{0}(s,\nu_{s}^{t,\nu},P_{s}\nu_{s}^{t,\nu}+\varphi_{s}^{t,\nu})\nabla\nu_{s}^{t,\nu}\right.\\
&-\partial_{\varphi}f_{0}(s,\nu_{s}^{t,\nu},P_{s}\nu_{s}^{t,\nu}+\varphi_{s}^{t,\nu}) \nabla\varphi_{s}^{t,\nu}-P_{s}\partial_{\nu}b_{0}(s,\nu_{t}^{t,\nu},P_{s}\nu_{s}^{t,\nu}+\varphi_{s}^{t,\nu})\nabla\nu_{s}^{t,\nu}\\
&\left.-P_{s}\partial_{\varphi}b_{0}(s,\nu_{s}^{t,\nu},P_{s}\nu_{s}^{t,\nu}+\varphi_{s}^{t,\nu})\nabla\varphi_{s}^{t,\nu}\right]ds,\\
\nabla \nu_{t}^{t,\nu} =& 1, \nabla \varphi_T^{t, \nu}=h_{2}^{\prime}\left(\nu_T^{t, \nu}\right) \nabla \nu_T^{t, \nu}.
\end{aligned}\right.\label{formal differentiation}
\end{equation}
\begin{Theorem} Under assumptions $(A1)-(A4)$, the function $\Phi \in C^{1}([0, T] \times$ $\mathbb{R};\mathbb{R}^{n})$ defined as \eqref{phi decoupling} is the unique classical solution to \eqref{phi pde} with bounded $\partial_\nu \Phi $.
\end{Theorem}
\begin{proof} From Theorem \ref{markovian phi theorem}, we know
\begin{equation}
     \Phi\left(t, \nu\right)=Y_{t}^{t,\eta}-P_{t}X_{t}^{t,\eta}, \quad \forall t\in[0,T],\label{phi definition}
 \end{equation}
 where $X^{t,\eta},Y^{t,\eta}$ are the first two components of the unique solution of mean field FBSDE \eqref{mean field fbsde} with initial condition $\eta\in L^{2}(\Omega,\mathcal{F}_{t},\mathbb{P};\mathbb{R})$ satisfying $\mathbb{E}[\eta] = \nu$ and $P$ is the unique solution of \eqref{p equation}.
 Taking expectation on the both side of
\eqref{phi definition}, we get
\begin{equation*}
\Phi(t,\nu)=
    \mathbb{E}[Y_{t}^{t,\eta}] - P_{t}\nu, \quad \forall t\in[0,T].
\end{equation*}
According to Lemma \ref{small bound} and Lemma \ref{P bound}, we obtain 
\begin{equation}
    \left|\Phi(t,\nu_{1})-\Phi(t,\nu_{2})\right| \leq C\left|\nu_{1}-\nu_{2}\right|, \forall \nu_{1},\nu_{2} \in \mathbb{R},\quad \forall t\in[0,T],\label{phi uniform}
\end{equation}
where C only depending on $n,K,T$. 

Now we would like to show that $\Phi\in C^{1}([0,T]\times \mathbb{R};\mathbb{R}^{n})$ with bounded $\partial_{\nu}\Phi$. 
 Let us consider the linear FBODE \eqref{formal differentiation} on $\left[t, T\right]$ for any $t\in[0, T)$. Note that under assumptions (A1)-(A4), all the coefficients in FBODE \eqref{formal differentiation} are bounded by some chosen $C_1 \geq C$. By standard FBODE arguments, there exists some $\delta_0>0$ depending on $C_{1}$ such that the FBODE \eqref{formal differentiation} is well-posed on $\left[T-\delta_0, T\right]$, which implies that $\partial_{\nu} \Phi \in C^{0}([T-\delta,T]\times \mathbb{R};\mathbb{R}^{n})$. Combined with \eqref{phi uniform}, following standard arguments, we obtain that $\Phi\in C^{1}([T-\delta,T]\times \mathbb{R};\mathbb{R}^{n})$. We then consider the FBODE \eqref{formal differentiation} with $h_{2}(\cdot)$ replaced by $\Phi\left(T-\delta_0, \cdot\right)$. According to \eqref{phi uniform}, the FBODE \eqref{formal differentiation} is also well-posed on $\left[T-2 \delta_0, T-\delta_0\right]$. Repeating this procedure backwardly finitely many
times, we are able to show that the FBODE \eqref{formal differentiation} is well-posed on $\left[t, T\right]$ for any $t\in[0,T]$ and  $\Phi \in C^{1}([0, T] \times \mathbb{R};\mathbb{R}^{n})$ with bounded $\partial_{\nu}\Phi$. Consequently, $\Phi$ is a classical solution of PDE \eqref{phi pde}.\\
\textbf{Uniqueness:} Suppose that $\tilde{\Phi}\in C^{1}([0, T] \times \mathbb{R};\mathbb{R}^{n})$ is another classical solution to \eqref{phi pde} with bounded $\partial_\nu \tilde{\Phi}$. For any $\left(t, \nu\right) \in[0, T] \times \mathbb{R}$, we first consider the following well-posed ODE
\begin{equation*}
\left\{\begin{aligned}
d \tilde{\nu}_s^{t, \nu} &= \left[ (b_{1}(s) +b_{2}(s) P_{s})\tilde{\nu}_s^{t, \nu} + b_{2}(s) \tilde{\Phi}(s,\tilde{\nu}_s^{t, \nu}) +b_{0}(s,\tilde{\nu}_s^{t, \nu},P_{s}\tilde{\nu}_s^{t, \nu}+\tilde{\Phi}(s,\tilde{\nu}_s^{t, \nu}))\right]ds,\\
\tilde{\nu}_{t}^{t, \nu} &=   \nu.
\end{aligned}\right.
\end{equation*}
Let $\tilde{\varphi}_s^{t, \nu}:=\tilde{\Phi}\left(s, \tilde{\nu}_s^{t, \nu}\right)$. Since $\tilde{\Phi}$ is a classical solution to \eqref{phi pde}, it can be easily checked that $\tilde{\varphi}^{t, \nu}$ solves the backward ordinary differential equation in \eqref{markovian phi}. Therefore we have verified that $\left(\tilde{\nu}^{t, \nu}, \tilde{\varphi}^{t, \nu}\right)$ is a solution to  FBODE \eqref{markovian phi}. Therefore, the uniqueness result follows by the well-posedness of the FBODE \eqref{markovian phi}.
\end{proof}

\begin{Theorem}
 Let Assumptions $(A1)-(A4)$ hold, then the function 
 \begin{equation}
U(t,x,\nu):=P_{t}x+\Phi(t,\nu)\label{U definition}
 \end{equation}
is the unique classical solution to the master equation \eqref{master equation} with bounded $\partial_{x}U,\partial_{\nu}U$.
\end{Theorem}
\begin{proof}
\textbf{Existence:} First, by Theorem \ref{markovian phi theorem}, we know for $t\in[0,T]$, $s\in[t,T]$, 
\begin{equation}
    Y_{s}^{t,\eta} = P_{s}X^{t,\eta}_{s} + \varphi_{s}^{t,\nu}= P_{s}X^{t,\eta}_{s}+\Phi(s,\mathbb{E}[X^{t,\eta}_{s}]) = U(t,X^{t,\eta}_{s},\mathbb{E}[X_{s}^{t,\eta}]),\label{U decoupling}
\end{equation}
where $\eta \in L^{2}(\Omega,\mathcal{F}_{t},\mathbb{P};\mathbb{R})$ with $\mathbb{E}[\eta] = \nu$. Therefore, the function $U(t,x,\nu)$ is the decoupling field of mean field FBSDE \eqref{mean field fbsde}. \\
\indent Next, we verify the decoupling field $U$ satisfies the master equation \eqref{master equation}. We first check that $U$ satisfies the terminal condition
\begin{equation*}
U(T, x, \nu)=P_T x+\Phi(T, \nu)=h_{1}x+h_{2}(\nu).
\end{equation*}
Moreover, it follows from \eqref{U decoupling} by setting $s=t$ and taking expectation that
\begin{equation}
P_{t}\nu+\Phi(t,\nu)=\mathbb{E}[U(t,\eta,\nu)],\label{u expectation}
\end{equation}
where $\eta \in L^{2}(\Omega,\mathcal{F}_{t},\mathbb{P};\mathbb{R})$ with $\mathbb{E}[\eta] = \nu$.\\
Recalling \eqref{p equation} and \eqref{phi pde}, we obtain
\begin{equation}
\begin{aligned}
 \partial_t U(t, x, \nu)&=\partial_tP_t x+\partial_t \Phi(t, \nu) \\
& =-\left((b_{2}(t) P_{t})P_{t}+f_{2}(t) P_{t}+b_{1}(t)P_{t}+f_{1}(t)\right) x-\partial _{\nu}\Phi(t,\nu)\left[(b_{1}(t)+b_{2}(t)P_{t})\nu\right.\\
&\left.\quad +b_{2}(t)\Phi(t,\nu)+b_{0}(t,\nu,P_{t}\nu+\Phi(t,\nu))\right]
     -(b_{2}(t))\Phi(t,\nu))P_{t}-f_{2}(t)\Phi(t,\nu)\\
     &\quad -f_{0}(t,\nu,P_{t}\nu+\Phi(t,\nu))-P_{t}b_{0}(t,\nu,P_{t}\nu+\Phi(t,\nu)).
\end{aligned}\label{master equation2}
\end{equation}
Moreover, we have
\begin{equation}
\partial_x U(t, x, \nu)=P_t, \quad \partial_{x x} U(t, x, \nu)=0, \quad \partial_\nu U(t, x, \nu)=\partial_\nu \Phi(t, \nu)
\end{equation}
are all bounded. Plugging the above terms into \eqref{master equation} and using \eqref{U definition}, \eqref{u expectation}, it is straightforward to show that $U$ is a classical solution to the master equation \eqref{master equation}.\\
\textbf{Uniqueness:} We recall that the solution $U$ to \eqref{master equation} serves as the decoupling field of mean field FBSDE \eqref{mean field fbsde}. Following the uniqueness argument in Theorem \ref{global existence}, the well-posedness of \eqref{mean field fbsde} implies the uniqueness of a solution to the master equation \eqref{master equation}.
\end{proof}
\bibliographystyle{siam}
\bibliography{bib}
\end{document}